\documentclass[12pt]{amsart}
\usepackage{amsmath,amssymb,amsbsy,amsfonts,latexsym,amsopn,amstext,cite,
                                               amsxtra,euscript,amscd,bm}
\usepackage{url}

\usepackage{mathrsfs}

\usepackage{color}
\usepackage[colorlinks,linkcolor=blue,anchorcolor=blue,citecolor=blue,backref=page]{hyperref}
\usepackage{color}
\usepackage{graphics,epsfig}
\usepackage{graphicx}
\usepackage{float}
\usepackage{epstopdf}
\hypersetup{breaklinks=true}

\usepackage[np]{numprint}
\npdecimalsign{\ensuremath{.}}

\usepackage{bibentry}

\usepackage[english]{babel}
\usepackage{mathtools}
\usepackage{todonotes}
\usepackage{url}
\usepackage[colorlinks,linkcolor=blue,anchorcolor=blue,citecolor=blue,backref=page]{hyperref}

\usepackage[norefs,nocites]{refcheck}

\DeclareMathOperator*\uplim{\overline{lim}}

\begin{document}

\newcommand{\bs}{\boldsymbol}
%\def \a{\alpha} \def \b{\beta} \def \d{\delta} \def
% \e{\varepsilon} \def \g{\gamma} \def \k{\kappa} 
% \def \l{\lambda} \def \s{\sigma} \def \t{\theta} \def \z{\zeta}

\newcommand{\mb}{\mathbb}

\newtheorem{theorem}{Theorem}
\newtheorem{lemma}[theorem]{Lemma}
\newtheorem{claim}[theorem]{Claim}
\newtheorem{cor}[theorem]{Corollary}
\newtheorem{conj}[theorem]{Conjecture}
\newtheorem{prop}[theorem]{Proposition}
\newtheorem{definition}[theorem]{Definition}
\newtheorem{question}[theorem]{Question}
\newtheorem{example}[theorem]{Example}
\newcommand{\hh}{{{\mathrm h}}}
\newtheorem{remark}[theorem]{Remark}

\numberwithin{equation}{section}
\numberwithin{theorem}{section}
\numberwithin{table}{section}
\numberwithin{figure}{section}

\def\sssum{\mathop{\sum\!\sum\!\sum}}
\def\ssum{\mathop{\sum\ldots \sum}}
\def\iint{\mathop{\int\ldots \int}}

\newcommand{\diam}{\operatorname{diam}}

\def\squareforqed{\hbox{\rlap{$\sqcap$}$\sqcup$}}
\def\qed{\ifmmode\squareforqed\else{\unskip\nobreak\hfil
\penalty50\hskip1em \nobreak\hfil\squareforqed
\parfillskip=0pt\finalhyphendemerits=0\endgraf}\fi}%%

%  use the AMS-Euler Fraktur fonts
%%%%%%%%%%%%%%%%%%%%%%%%%%%%%%%%%%
\newfont{\teneufm}{eufm10}
\newfont{\seveneufm}{eufm7}
\newfont{\fiveeufm}{eufm5}
%%%%%%%%%%%%%%%%%%%%%%%%%%%%%%%%%
%
%  allow automatic size selection in math mode
%
%%%%%%%%%%%%%%%%%%%%%%%%%%%%%%%%%
\newfam\eufmfam
     \textfont\eufmfam=\teneufm
\scriptfont\eufmfam=\seveneufm
     \scriptscriptfont\eufmfam=\fiveeufm
%%%%%%%%%%%%%%%%%%%%%%%%%%%%%%%%%
%
%  \frak works on a single symbol at a time...
%
\def\frak#1{{\fam\eufmfam\relax#1}}

\newcommand{\bflambda}{{\boldsymbol{\lambda}}}
\newcommand{\bfmu}{{\boldsymbol{\mu}}}
\newcommand{\bfxi}{{\boldsymbol{\eta}}}
\newcommand{\bfrho}{{\boldsymbol{\rho}}}

\def\eps{\varepsilon}

\def\fK{\mathfrak K}
\def\fT{\mathfrak{T}}
\def\fL{\mathfrak L}
\def\fR{\mathfrak R}

\def\fA{{\mathfrak A}}
\def\fB{{\mathfrak B}}
\def\fC{{\mathfrak C}}
\def\fM{{\mathfrak M}}
\def\fS{{\mathfrak  S}}
\def\fU{{\mathfrak U}}
\def\fQ{{\mathfrak Q}}

\def\sssum{\mathop{\sum\!\sum\!\sum}}
\def\ssum{\mathop{\sum\ldots \sum}}
\def\dsum{\mathop{\quad \sum \qquad \sum}}
\def\iint{\mathop{\int\ldots \int}}
 
\def\T {\mathsf {T}}
\def\Tor{\mathsf{T}_d}
\def\Tore{\widetilde{\mathrm{T}}_{d} }

\def\sM {\mathsf {M}}
\def\sL {\mathsf {L}}
\def\sK {\mathsf {K}}
\def\sP {\mathsf {P}}

\def\ss{\mathsf {s}}

\def\Kmnd{\cK_d(m,n)}
\def\Kmnp{\cK_p(m,n)}
\def\Kmnq{\cK_q(m,n)}

\def \btheta{\bm{\vartheta}}
 
\def \balpha{\bm{\alpha}}
\def \bbeta{\bm{\beta}}
\def \bgamma{\bm{\gamma}}
\def \bdelta{\bm{\delta}}
\def \bzeta{\bm{\zeta}}
\def \blambda{\bm{\lambda}}
\def \bchi{\bm{\chi}}
\def \bphi{\bm{\varphi}}
\def \bpsi{\bm{\psi}}
\def \bxi{\bm{\xi}}
\def \bnu{\bm{\nu}}
\def \bomega{\bm{\omega}}

\def \bell{\bm{\ell}}

\def\eqref#1{(\ref{#1})}

\def\vec#1{\mathbf{#1}}

\newcommand{\abs}[1]{\left| #1 \right|}

\def\Zq{\mathbb{Z}_q}
\def\Zqx{\mathbb{Z}_q^*}
\def\Zd{\mathbb{Z}_d}
\def\Zdx{\mathbb{Z}_d^*}
\def\Zf{\mathbb{Z}_f}
\def\Zfx{\mathbb{Z}_f^*}
\def\Zp{\mathbb{Z}_p}
\def\Zpx{\mathbb{Z}_p^*}
\def\cM{\mathcal M}
\def\cE{\mathcal E}
\def\cH{\mathcal H}

\def\le{\leqslant}

\def\ge{\geqslant}

\def\sfB{\mathsf {B}}
\def\sfC{\mathsf {C}}
\def\sfG{\mathsf {G}}
\def\sfS{S}
\def\sfI{\mathsf {I}}
\def\L{\mathsf {L}}
\def\FF{\mathsf {F}}

\def\sE {\mathscr{E}}
\def\sF {\mathscr{F}}
\def\sG {\mathscr{G}}
\def\sQ {\mathscr{Q}}
\def\sS {\mathscr{S}}

%%%%%%%%%%%%%%%%%%%%%%%%%
% Alphabet calligraphie %
%%%%%%%%%%%%%%%%%%%%%%%%%
\def\cA{{\mathcal A}}
\def\cB{{\mathcal B}}
\def\cC{{\mathcal C}}
\def\cD{{\mathcal D}}
\def\cE{{\mathcal E}}
\def\cF{{\mathcal F}}
\def\cG{{\mathcal G}}
\def\cH{{\mathcal H}}
\def\cI{{\mathcal I}}
\def\cJ{{\mathcal J}}
\def\cK{{\mathcal K}}
\def\cL{{\mathcal L}}
\def\cM{{\mathcal M}}
\def\cN{{\mathcal N}}
\def\cO{{\mathcal O}}
\def\cP{{\mathcal P}}
\def\cQ{{\mathcal Q}}
\def\cR{{\mathcal R}}
\def\cS{{\mathcal S}}
\def\cT{{\mathcal T}}
\def\cU{{\mathcal U}}
\def\cV{{\mathcal V}}
\def\cW{{\mathcal W}}
\def\cX{{\mathcal X}}
\def\cY{{\mathcal Y}}
\def\cZ{{\mathcal Z}}
\newcommand{\rmod}[1]{\: \mbox{mod} \: #1}

\def\cg{{\mathcal g}}

\def\vy{\mathbf y}
\def\vr{\mathbf r}
\def\vx{\mathbf x}
\def\va{\mathbf a}
\def\vb{\mathbf b}
\def\vc{\mathbf c}
\def\ve{\mathbf e}
\def\vf{\mathbf f}
\def\vg{\mathbf g}
\def\vh{\mathbf h}
\def\vk{\mathbf k}
\def\vm{\mathbf m}
\def\vz{\mathbf z}
\def\vu{\mathbf u}
\def\vv{\mathbf v}

\def\e{{\mathbf{\,e}}}
\def\ep{{\mathbf{\,e}}_p}
\def\eq{{\mathbf{\,e}}_q}

\def\Tr{{\mathrm{Tr}}}
\def\Nm{{\mathrm{Nm}}}

 \def\SS{{\mathbf{S}}}

\def\lcm{{\mathrm{lcm}}}

\def\pow{\tau}

 \def\0{{\mathbf{0}}}

\def\({\left(}
\def\){\right)}
\def\l|{\left|}
\def\r|{\right|}
\def\fl#1{\left\lfloor#1\right\rfloor}
\def\rf#1{\left\lceil#1\right\rceil}
\def\sumstar#1{\mathop{\sum\vphantom|^{\!\!*}\,}_{#1}}

\def\mand{\qquad \mbox{and} \qquad}

\def\tblue#1{\begin{color}{blue}{{#1}}\end{color}}

%%%%%%%%%%%%%%%%%%%%%%%%%%%%%%%%%%%%%%%%%%%%%%%%%%%%%%%%
%%%%%%%%%%%%%%%%%%%%%%%%%%%%%%%%%%%%%%%%%%%%%%%%%%%%%%%%
%%%%%%%%%%%%%%%%%%%%%%%%%%%%%%%%%%%%%%%%%%%%%%%%%%%%%%%%
%%%%%%%%%%%%%%%%%%%%%%%%%%%%%%%%%%%%%%%%%%%%%%%%%%%%%%%%

%%%%%%%  END OF STANDARD STUFF %%%%%%%%%

%%%%%%%%%%%%%%%%%%%%%%%%%%%%%%%%%%%%%%%%%%%%%%%%%%%%%%%%
%%%%%%%%%%%%%%%%%%%%%%%%%%%%%%%%%%%%%%%%%%%%%%%%%%%%%%%%
%%%%%%%%%%%%%%%%%%%%%%%%%%%%%%%%%%%%%%%%%%%%%%%%%%%%%%%%
%%%%%%%%%%%%%%%%%%%%%%%%%%%%%%%%%%%%%%%%%%%%%%%%%%%%%%%
%%%%%%%%%%%
%%% Spell

\hyphenation{re-pub-lished}

\mathsurround=1pt

\def\bfdefault{b}

\def \F{{\mathbb F}}
\def \K{{\mathbb K}}
\def \N{{\mathbb N}}
\def \Z{{\mathbb Z}}
\def \P{{\mathbb P}}
\def \Q{{\mathbb Q}}
\def \R{{\mathbb R}}
\def \C{{\mathbb C}}
\def\Fp{\F_p}
\def \fp{\Fp^*}

 \def \xbar{\overline x}

%%\title[]{On Hausdorff dimension of large sets of Weyl sums} 

\title{Large Weyl sums and Hausdorff dimension}

\author[R. C. Baker] {Roger C.~Baker} 
\address{Department of Mathematics, Brigham Young University, 
Provo, UT 84602, USA} 
\email{baker@math.byu.edu}

 \author[C. Chen] {Changhao Chen}
\address{Center  for Pure Mathematics, School of Mathematical Sciences, Anhui University, Hefei 230601, China}
\email{chench@ahu.edu.cn}

 \author[I. E. Shparlinski] {Igor E. Shparlinski}
\address{Department of Pure Mathematics, University of New South Wales,
Sydney, NSW 2052, Australia}
\email{igor.shparlinski@unsw.edu.au}

\begin{abstract}  We obtain the exact value of the Hausdorff dimension of the set 
of coefficients of Gauss sums which  for a given $\alpha \in (1/2,1)$  achieve the order at least $N^{\alpha}$ 
for infinitely many sum lengths $N$. For Weyl sums with polynomials of degree $d\ge 3$ we obtain a 
new upper bound on the Hausdorff dimension of the set of polynomial coefficients corresponding to 
large values of Weyl sums. Our methods also work for monomial sums, match the previously 
known lower bounds, just giving exact value for the corresponding Hausdorff dimension  
when $\alpha$ is close to $1$. We also obtain a nearly tight  bound in a similar question 
with arbitrary integer sequences of polynomial growth. 
\end{abstract}


\maketitle

\tableofcontents 

\section{Introduction}

\subsection{Set-up and motivation}  For $\vx = (x_1, \ldots, x_d)\in \T_d$ 
where   
$$
\T_d = [0,1]^d
$$
is the $d$-dimensional unit cube, we define the \textit{Weyl sums} 
of length $N$ as 
$$
\sfS_{d}(\vx; N)  = \sum_{n=1}^{N} \e\(x_1n + \ldots + x_dn^d\)
$$
where  $\e(z) = \exp(2\pi i z)$. These sums were originally introduced by Weyl to study equidistribution of fractional parts of polynomials and later find their applications to the circle method and Riemann zeta function. 

For many applications of Weyl sums, the key problem is to estimate the size of the sum $\sfS_{d}(\vx; N)$.  There are often three kinds of estimates of Weyl sums, namely individual bounds, mean value bounds and almost all bounds.  Despite more than a century since these sums were introduced, their behaviour for individual values of $\vx$ is not well understood, see~\cite{Brud,BD, Kerr}. 

Much more is known about the average behaviour of $\sfS_d(\vx;N)$. The recent advances of  Bourgain, Demeter and Guth~\cite{BDG} (for $d \geqslant 4$) 
and Wooley~\cite{Wool-3} (for $d=3$)  (see also~\cite{Wool}) towards  the optimal form Vinogradov mean value theorem imply the estimate
$$
 N^{s(d)} \le \int_{\Tor} |\sfS_d(\vx; N)|^{2s(d)}d\vx \leqslant  N^{s(d)+o(1)}, 
$$
where   
$$
s(d)=d(d+1)/2
$$
and is best possible up to $o(1)$ in the exponent of $N$.

We study exceptional sets of $\vx \in \Tor$, which generate abnormally large  
Weyl sums $\sfS_d(\vx; N)$. 

The first results concerning the almost all behaviour of Weyl sums are due to Hardy and Littlewood~\cite{HL1} who have estimated  the following special sums  
$$
\sfG(x,N) = \sum_{n=1}^N \e\(x n^2\),
$$ 
in terms of the continued fraction expansion of $x$. Among other things Hardy and Littlewood~\cite{HL1} proved that for almost all $x\in \T$,
$$
\left|\sfG(x,N)\right|\le N^{1/2+o(1)}, \quad \text{as} \quad N\rightarrow \infty.
$$
Their idea has  been expanded upon by Fiedler, Jurkat and K\"orner~\cite[Theorem~2]{FJK} who give  the following optimal lower and upper bounds.
 Suppose that  $\{f(n)\}_{n=1}^{\infty}$ is a non-decreasing sequence of positive numbers. Then for almost all  $x\in \T$  one has 
\begin{equation}
\label{eq:G}
\uplim_{N\rightarrow  \infty} \frac{  |\sfG(x,N)|}{\sqrt{N} f(N)}<\infty\quad  \Longleftrightarrow \quad \sum_{n=1}^{\infty} \frac{1}{n f(n)^{4}} <\infty.
\end{equation}

For the sums $\sfS_2(\vx; N)$, $\vx \in \T_2$, Fedotov and Klopp~\cite{FK} have given a similar result, however 
adding the term $\e(x_1n)$ leads to more cancellations in the sums $\sfS_2(\vx; N)$. Suppose that  $\{g(n)\}_{n=1}^{\infty}$ is a non-decreasing sequence of positive numbers. Then for almost all  $\vx\in \T_2$  one has 
\begin{equation}
\label{eq:QW}
\uplim_{N\rightarrow  \infty} \frac{  \left |\sfS_2(\vx; N)\right |}{\sqrt{N} g(\ln  N)}<\infty \quad  \Longleftrightarrow \quad  \sum_{n=1}^{\infty} \frac{1}{ g(n)^{6}} <\infty.
\end{equation}

It is natural to expect that analogues  of~\eqref{eq:G},~\eqref{eq:QW} hold for Weyl sums $\sfS_d(\vx; N)$ with any $d\ge 3$, however this question seems to be still open.  However, we have the following nearly sharp bounds.  For $d\ge 3$, Chen and Shparlinski~\cite{ChSh-AM,ChSh-IMRN} have  shown in two different ways  in~\cite[Appendix~A]{ChSh-AM},  and~\cite[Theorem~2.1]{ChSh-IMRN}  that  for almost all $\vx\in \Tor$  one has   
 \begin{equation}
\label{eq:Weyl-U}
  \left |\sfS_d(\vx; N)\right |\leqslant  N^{1/2+o(1)} \quad \text{as} \quad N\rightarrow  \infty.
\end{equation}
Recently, Chen, Kerr, Maynard and Shparlinski~\cite[Theorem~2.3]{CKMS} have shown that the exponent $1/2$ is optimal, that is, there exists a constant $c>0$ such that for almost all $\vx\in \T_d$, the inequality 
\begin{equation}
\label{eq:Weyl-L}
|\sfS_{d}(\vx; N)|\ge cN^{1/2} 
\end{equation}
holds for infinitely many $N$.

This motivates our study of  the ``the exceptional sets" of Weyl sums.  Precisely,  for $1/2<\alpha<1$  define
$$
\cE_{d, \alpha}=\{\vx\in \Tor:~|\sfS_{d}(\vx; N)|\geqslant N^{\alpha} \text{ for infinitely many } N\in \N\}.
$$ 

Chen and Shparlinski~\cite[Theorem~1.3]{ChSh-AM} show that for any $d\ge 2$ and $1/2<\alpha<1$ the set $\cE_{d, \alpha}$ is of second category in the sense of Baire, and  the proof of~\cite[Theorem~1.3]{ChSh-AM} implies that the set $\cE_{d, \alpha}$ is a dense subset of $\T_d$. Therefore, the {\it Minkowski dimension\/} (or box dimension) 
of $\cE_{d, \alpha}$ is $d$.  See~\cite{Falconer} for more details on the Minkowski dimension.

The above results~\eqref{eq:QW} and~\eqref{eq:Weyl-U} imply that $\cE_{d, \alpha}$ is of zero Lebesgue measure for all $d\ge 2$ and any $\alpha\in (1/2, 1)$. 
For sets of  Lebesgue measure zero, it is common to use  the {\it Hausdorff dimension\/} to describe their size and structure, and we are going to estimate the  Hausdorff dimension of the set $\cE_{d, \alpha}$ for any $d\ge 2$ and any $\alpha\in (1/2, 1)$.  We first recall the formal definition of Hausdorff dimension, and we refer to~\cite{Falconer, Mattila1995} for 
more details.

\begin{definition} 
\label{def:Hausdorff}
The  Hausdorff dimension of a set $\cF\subseteq \R^{d}$ is defined as 
\begin{align*}
\dim \cF=\inf\Bigl\{s>0:~\forall \, & \eps>0,~\exists \, \{ \cU_i \}_{i=1}^{\infty}, \ \cU_i \subseteq \R^{d},\\
&  \text{such that } \cF\subseteq \bigcup_{i=1}^{\infty} \cU_i \text{ and } \sum_{i=1}^{\infty}\(\diam\cU_i\)^{s}<\eps \Bigr\},
\end{align*}
where 
$$
\diam \cU = \sup\{\| u-v\|:~u,v \in \cU\}
$$
and  $\|w\|$ is the Euclidean norm in $\R^{d}$. 
\end{definition}

We remark that we could also define the set $\cE_{d, \alpha}$ for $\alpha\in (0, 1/2]$. However, by~\eqref{eq:Weyl-L}
the set $\cE_{d, \alpha}$ is of full Lebesgue measure. This, by [17, Theorem 2.5] and the definition of the Hausdorff dimension, is enough to conclude that 
$$
\dim \cE_{d, \alpha}=d.
$$

For an integer $d\ge 3$ and  real    
$\alpha\in (1/2,1)$ some explicit  upper and lower bounds on $\dim \cE_{d, \alpha}$ have been given
in~\cite{ChSh-AM,ChSh-JNT,ChSh-IMRN}. In particular, for any $\alpha\in (1/2, 1)$, there are explicit functions $\mathfrak{l}(d, \alpha), \mathfrak{u}(d, \alpha)$ such that 
$$
0<\mathfrak{l}(d, \alpha)\le \dim \cE_{d, \alpha}\le \mathfrak{u}(d, \alpha)<d.
$$ 
We show more details in the following. For $d\ge 2$, let 
$$
\kappa_d =  \max_{\nu =1, \ldots, d} \min\left\{ \frac{1}{2\nu} ,  \frac{1}{2d-\nu} \right \}.
$$
For  each $1/2<\alpha<1$ and any cube $\fQ \subseteq \Tor$   we have the following lower bounds of $\dim \cE_{d, \alpha}$:
\begin{itemize}
\item[(i)] for $d=2$, 
$$
\dim \cE_{2, \alpha} \cap \fQ  \geqslant 3(1-\alpha)/2; 
$$
\item[(ii)] for $d\geqslant 3$, 
$$
\dim \cE_{d, \alpha} \cap \fQ \geqslant    2 \kappa_d (1-\alpha) .
$$
\end{itemize}

For the upper bound  of $\dim \cE_{d, \alpha}$ with $d\ge 2$ and $1/2<\alpha<1$, we have 
$$
\dim \cE_{d, \alpha}\le \mathfrak{u}(d, \alpha),
$$
where 
$$
\mathfrak{u}(d, \alpha)=\min_{k=0, \ldots, d-1} \frac{(2d^{2}+4d)(1-\alpha)+k(k+1)}{4-2\alpha+2k}.
$$
It is not hard to show $\mathfrak{u}(d, \alpha)<d$ for any $\alpha\in (1/2, 1)$.

In fact for $\alpha \rightarrow 1$ the behaviour of 
$\dim \cE_{d, \alpha}$ is understood reasonably well as  a combination of the above mentioned  lower  and upper bounds, implies that  there are  positive constants $c_1(d), c_2(d)$ such that 
$$
c_1(d) \le \liminf_{\alpha \rightarrow 1}\, (1-\alpha)^{-1}\dim \cE_{d, \alpha}\le\limsup_{\alpha \rightarrow 1} \, (1-\alpha)^{-1}\dim \cE_{d, \alpha}\le c_2(d).
$$

For $\alpha\in (1/2, 1)$,  some heuristic arguments have been given in~\cite{CKMS} towards the following:

\begin{conj}
\label{conj:HD}
For any $\alpha\in (1/2, 1)$, the set  $\cE_{d, \alpha}$ is of Hausdorff dimension 
$$
\dim \cE_{d, \alpha}=\min_{j=1, \ldots, d}\frac{d+1+j\vartheta_j -\sum_{i=1}^{j} \vartheta_i}{1+\vartheta_j} ,
$$
where 
$$
\vartheta_i=\frac{i}{2(1-\alpha)}-1, \qquad i=1, \ldots, d.
$$ 
\end{conj}

\subsection{New results and methods}  
In this paper, we confirm the Conjecture~\ref{conj:HD} for $d=2$, and we obtain new upper  bounds of $\dim \cE_{d,\alpha}$ when  $d\ge 3$ and $\alpha$ is close to $1$. Moreover, we also consider the following one parametric family of  exponential sums. Namely, for a real sequence $f(n)$, $n\in \N$,
$ x\in \T$ and $N\in \N$ we denote 
\begin{equation}
\label{eq:sum V} 
V_f(x; N)=\sum_{n=1}^{N}\e\(xf(n)\). 
\end{equation}

Chen and Shparlinski~\cite[Corollary 2.2]{ChSh-IMRN} shows that for any polynomial $f \in \Z[X]$ with $\deg f\ge 2$ we have for almost all $x\in \T$,
$$
\left| V_f(x; N)\right| \le N^{1/2+o(1)} \qquad \text{as} \ N\rightarrow  \infty.
$$

Similarly to the definition of $\cE_{d, \alpha}$, for $\alpha\in (1/2, 1)$ we define the set 
$$
\cF_{f, \alpha}=\{x\in \T:~|V_f(x; N)|\ge N^{\alpha} \text{ for infinitely many $N\in \N$ }\}.
$$  
Perhaps the most interesting sums of this type are sums with monomials $xn^d$, in which 
case we denote this special quantity by $\sF_{d, \alpha}$, that is, 
$$\sF_{d, \alpha}=\left\{x\in \T:~\left|\sum_{n=1}^{N}\e\(xn^d\)\right|\ge N^{\alpha} \text{ for infinitely many $N\in \N$ }\right\}.
$$

Some lower bounds of $\dim \sF_{d, \alpha}$ have been obtained in~\cite[Theorem~1.7]{ChSh-AM}. 
In particular,~\cite[Theorem~1.7]{ChSh-AM}  implies that for $\alpha\in (1/2, 1)$
\begin{equation}
\label{eq:low Gauss} 
\dim \sF_{2, \alpha}\ge 2(1-\alpha),
\end{equation}
and for $d\ge 3$ and $\alpha \in [d/(d+2), 1)$,   
\begin{equation}
\label{eq:low monom} 
\dim \sF_{d, \alpha}\ge \(1+\frac{1}{d}\)(1-\alpha).
\end{equation}

Some heuristic arguments have been given in~\cite[Section~8]{CKMS},  suggesting that in a certain range of $\alpha$  we may have 
$$
\dim \sF_{d, \alpha} =   4(1-\alpha)/d,
$$
which is consistent with~\eqref{eq:low Gauss} and~\eqref{eq:low monom} for $d=2$ and $d=3$.

Moreover, for $\alpha\in (0, 1/2)$ the set  $\sF_{d, \alpha}$ is of positive Lebesgue measure (see~\cite{CKMS} for more details), and hence,
$$
\dim \sF_{d, \alpha} =1.
$$

It is very likely that for $f \in \Z[X]$ the bounds~\eqref{eq:low Gauss} and~\eqref{eq:low monom} can be extended to
the sets $\cF_{f, \alpha}$.

To obtain these results we develop two different approaches:
\begin{itemize}
\item For $\alpha$ close to $1$, we employ the classification of Baker~\cite{Bak0,Bak1} in the form given in~\cite{BCS}.
\item For smaller values of $\alpha$ (which means that  $\alpha$ is close to $1/2$),  and also for sums $V_f(x; N)$ with non-polynomial functions when   
the above classification is not available, we link Hausdorff dimension of the sets $\cE_{d, \alpha}$ and 
$\cF_{f, \alpha}$ to various mean value theorems. 
\end{itemize}
The above arguments are complemented by the use of the Frostman  Lemma
(see~\cite[Corollary~4.12]{Falconer}) and the G\'al--Koksma Theorem~\cite[Theorem~4]{GK}.

\subsection{Notation} 
Throughout the paper, the notations $U = O(V)$, 
$U \ll V$ and $ V\gg U$  are equivalent to $|U|\leqslant c V$ for some positive constant $c$, 
which throughout the paper may depend on the degree $d$, the growth rate $\tau$ and occasionally on the small real positive 
parameter $\varepsilon$ and the real parameter $t$.

For any quantity $V> 1$ we write $U = V^{o(1)}$ (as $V \rightarrow \infty$) to indicate a function of $V$ which 
satisfies $ V^{-\eps} \le |U| \le V^\eps$ for any $\eps> 0$, provided  that $V$ is large enough. One additional advantage 
of using $V^{o(1)}$ is that it absorbs $\log V$ and other similar quantities without changing  the whole 
expression.  

For $m \in \N$, we write $[m]$ to denote the set $\{0, 1, \ldots, m-1\}$.

\section{Results for  sets of very large sums}
 \label{sec:large sums}

\subsection{Multiparametric families of Gauss sums and Weyl sums}
Here we confirm the Conjecture~\ref{conj:HD} for $d=2$, that is, for the  Gauss sums  
 $$
G(\vx; N) = \sum_{n=1}^{N} \e\(x_1n + x_2n^2\),
$$
we improve the previous upper and lower bounds of~\cite{ChSh-AM,ChSh-JNT,ChSh-IMRN}, and  obtain the exact value of the Hausdorff dimension of $\cE_{2, \alpha}$.

\begin{theorem}
\label{thm:Gauss}
For any $\alpha\in (1/2, 1)$ we have 
$$
\dim \cE_{2, \alpha}= \begin{cases} 7/2-3\alpha & \text{if}\ 5/6 \ge \alpha >1/2,\\
 6(1-\alpha)  & \text{if}\  1> \alpha >5/6.
 \end{cases} 
$$
\end{theorem}

For $d\ge 3$, by applying the same idea as in the proof of Theorem~\ref{thm:Gauss} and combining some other new ideas, we obtain the following upper bound, which improves the previous bound of~\cite[Theorem~1.1]{ChSh-JNT} when $\alpha$ is close to $1$.  
For $d\ge 3$  we  always write 
\begin{equation}
\label{eq:D}
D=\min\{2^{d-1}, 2d(d-1)\}.
\end{equation}

\begin{theorem}
\label{thm:Weyl}
For $d\ge 3$ and any $\alpha\in (1-1/D, 1)$, where $D$ is given by~\eqref{eq:D}, we have 
$$
\dim \cE_{d, \alpha}\le  \min_{h=1, \ldots, d} \frac{(d^2+1)(1-\alpha)}{h}+\frac{h-1}{2}.
$$
\end{theorem}

\subsection{One parametric families of Weyl sums with real polynomials} 
We now obtain  upper bounds on $\dim \cF_{f, \alpha}$,  which in the case of monomial sums and large values of $\alpha$ matches the lower bounds~\eqref{eq:low Gauss} and~\eqref{eq:low monom}. Depending on using different methods, 
 we divide the results on $\dim \cF_{f, \alpha}$  into two subsections.

\begin{theorem}   
\label{thm:F_f} Let $f\in \R[X]$ be a polynomial of degree $d$.
For any $\alpha\in (1-1/D, 1)$,  where $D$ is given by~\eqref{eq:D},  we have 
$$
\dim \cF_{f, \alpha}\le
\begin{cases}  
2(1-\alpha)  & \text{if}\  d=2,\\
 \(1+1/d\)(1-\alpha) & \text{if}\ d  \ge 3.
 \end{cases} 
$$
\end{theorem}

Combining Theorem~\ref{thm:F_f} with~\eqref{eq:low Gauss} and~\eqref{eq:low monom} and noticing 
that for $d \ge 2$,  
$$
\frac{d}{d+2} \le 1-\frac{1}{D},
$$
we obtain the following exact values in the case of monomial sums. 

\begin{cor}  \label{cor:Ed} 
For any $\alpha\in (1-1/D, 1)$,  where $D$ is given by~\eqref{eq:D},  we have 
$$
\dim \sF_{d, \alpha} =
\begin{cases}  
2(1-\alpha)  & \text{if}\  d=2,\\
 \(1+1/d\)(1-\alpha) & \text{if}\ d  \ge 3.
 \end{cases} 
$$
\end{cor}

As we have mentioned, we believe that in the case $f \in \Z[X]$  the lower bounds~\eqref{eq:low Gauss} and~\eqref{eq:low monom}
and thus Corollary~\ref{cor:Ed}, 
can be extended to $\dim \cF_{f, \alpha}$.

We observe that Theorem~\ref{thm:F_f} can be applied to estimate $\dim \cE_{d, \alpha}\cap \cL$ where 
$\cL$ is a straight line in$\R^d$ passing through the origin. Precisely, let $\vv = (v_1, \ldots, v_d) \in \R^d$, $\vv\neq {\mathbf 0}$,  and 
$$
\cL_{\vv}=\{\lambda\vv:~ \lambda \in \T\}. 
$$
For $\vx\in \cL_{\vv}$ for some   $\lambda \in \T$ we have 
$$
\sfS_{d}(\vx; N)=\sum_{n=1}^N\e\( \lambda f(n)\), 
$$ 
where $ f(n) = v_1n+\ldots + v_dn^d$. 
It follows that 
$
\dim \( \cE_{d, \alpha}\cap \cL_{\vv} \) \le \dim \cF_{f, \alpha},
$
and by Theorem~\ref{thm:F_f} we derive the following explicit bound.

\begin{cor}
\label{cor:intersection}
Let $\vv=(v_1, \ldots, v_k, 0, \ldots, 0)\in \R^d$ with $v_k\neq 0$ and $v_j=0$ when $k<j\le d$. Let 
$$
D_k=\min\{2^{k-1}, 2k(k-1)\}.
$$
Then for any $\alpha\in (1-1/D_k, 1)$ we have 
$$
\dim \cE_{d, \alpha}\cap \cL_v \le 
\begin{cases}  
2(1-\alpha)  & \text{if}\  k=2,\\
 \(1+1/k\)(1-\alpha) & \text{if}\ k  \ge 3.
 \end{cases} 
$$
\end{cor}

We remark that Corollary~\ref{cor:Ed} implies that the bounds of Corollary~\ref{cor:intersection} is sharp in general when $\alpha$ is close to $1$. This follows by choosing $\vv=(v_1, \ldots, v_d)$ such that $v_k=1$ and $v_j=0$ when $j\neq k$.
 Moreover, motivated from the research on Diophantine approximation on manifolds 
(see, for instance~\cite[Chapter~9]{Harm}),  we pose the following general question.
 
 \begin{question} \label{que}
 Given an ``interesting'' surface $\Gamma \subseteq \R^d$, for example, an algebraic hypersurface or 
a smooth analytic  surface of given curvature, determine the Hausdorff dimension of the 
intersection  $\cE_{d, \alpha}\cap \Gamma$. 
\end{question}

We note that upper bounds on the means values of Weyl sums along various surfaces 
have been given in~\cite{ChSh-Surf,DeLa}. 
 
\subsection{Comparison} 
We observe that Theorem~\ref{thm:Weyl} improves the upper bound
\begin{equation}
\label{eq:Old Bound}
 \dim \cE_{d, \alpha}\le \min_{k=0, \ldots, d-1} \frac{(2d^{2}+4d)(1-\alpha)+k(k+1)}{4-2\alpha+2k}
\end{equation}
of~\cite[Theorem~1.1]{ChSh-JNT} for the range $\alpha\in (1-1/D, 1)$.   Indeed,  consider the functions 
$$
F(h, \beta) =  \frac{(d^2+1)\beta}{h}+\frac{h-1}{2}, 
\qquad 
G(k,  \beta) = \frac{(2d^{2}+4d) \beta+k(k+1)}{2 +2\beta+2k}
$$
and note that it is enough to verify that for $\beta \in (0, 1/D)$ and $h=1, \ldots, d$ we have 
\begin{equation}
\label{eq:F<G}
F(h, \beta)  <G(h-1,  \beta) . 
\end{equation}
Clearly, the inequality~\eqref{eq:F<G} is equivalent to 
$$
 \frac{(d^2+1)\beta}{h}+\frac{h-1}{2} <  \frac{(2d^{2}+4d) \beta+h(h-1)}{2\beta+2h}
$$
or
$$
\(2(d^2+1)\beta+h(h-1)\)\(\beta+h\) < \((2d^{2}+4d) \beta+h(h-1)\)h.
$$
Simplifying we obtain an equivalent inequality 
$$
2(d^2+1)\beta^2 < \beta\(4dh-h(h+1)\)
$$
and finally 
$$
2(d^2+1)\beta < 4dh-h(h+1). 
$$
Since $\beta < 1/D$ and the right hand side of above inequality is increasing with $h \in [1,d]$
it is only enough to check that 
$$
2(d^2+1)/D < 4d-2, 
$$
which can be numerically verified for $3 \le d \le 7$ and established via elementary calculus for $d \ge 8$.
This shows that~\eqref{eq:F<G} holds and thus Theorem~\ref{thm:Weyl} is stronger than~\eqref{eq:Old Bound}
for any $d \ge 3$ and all admissible values of $\alpha$. 

Clearly,  when $\alpha$ is close to $1$ then the choice of $h=1$ in Theorem~\ref{thm:Weyl}  is optimal.
However sometimes other choices of $h$ give better result. For example, if 
$\alpha \in (1-1/D, 1-1/(d^2+1)$ than then choice of $h=2$ is better than $h=1$. 
The above range is non-empty provided  that $d^2 + 1 \ge D$, and this happens when  $3 \le d \le 6$ only.
Moreover, for larger $d$ (say $d\ge 7$) and $\alpha\in (1-1/D, 1)$  the value $\alpha$ is quite near the value $1$, thus we may take $h=1$ only in Theorem~\ref{thm:Weyl} when $d\ge 7$.
Although for our applications here the values $h \ge 3$ are never used, we   present the argument in full generality
as we believe  it can be used to study Weyl sums with other polynomials.

\subsection{Ideas behind the proofs}  
We concentrate  on the ideas in the proof of Theorem~\ref{thm:Gauss}.   
Before doing this, we remark that  a similar argument to the proof of the upper bound on $\dim \cE_{2, \alpha}$  also implies  an upper bound on $\dim \cE_{d, \alpha}$ and  $\dim \cF_{f, \alpha}$  in Theorems~\ref{thm:Weyl} and~\ref{thm:F_f}. However, the idea for obtaining the lower bound on $\dim \cE_{2, \alpha}$ does not work for $d\ge 3$. The main reason is that we do not have a version of Lemma~\ref{lem:Gauss} when $d\ge 3$ and in fact 
for a prime $p \equiv 2 \pmod 3$ there are many vectors $(a, b, c)\in [p]^3$ such that 
$$
\sum_{n=1}^p\e_p(an+bn^2+cn^3)=0,
$$
see~\cite[Remark~2.8]{ChSh-AM} for more details. 

\subsubsection*{Upper bound:} Our argument is based on a combination of the Frostman  Lemma
(see~\cite[Corollary~4.12]{Falconer}) and the G\'al--Koksma Theorem~\cite[Theorem~4]{GK}, 
which are presented in Section~\ref{sec:FL GKT}. 

 For any $t<\dim \cE_{2,\alpha}$, by the \textit{Frostman Lemma} (see~\cite[Corollary 4.12]{Falconer} 
 or Lemma~\ref{lem:Frostman} below),  there exists a Radon measure $\mu$ on $\cE_{2,\alpha}$ with 
$$
\mu(\cE_{2, \alpha})>0 \quad \text{ and } \quad  \mu(B(\vx, r))\ll r^{t}  
$$
for all $\vx$ and $r>0$. Applying the description of    Baker~\cite[Theorem~3]{Bak0}   of  the structure of large Gauss sums, we obtain the following type $L^{\rho}$ bound: for any $M, N\in \N$ we have  
$$
\int_{\T_2} \left|  \sum_{n=M+1}^{M+N}\e(x_1n+x_2n^{2}) \right|^{\rho} d\mu(\vx)\le N^{s_1+o(1)}(N+M)^{s_2}, 
$$
where the exponents $s_1$ and $s_2$ depend on $\rho$ and $t$. By a result 
of G\'{a}l and Koksma~\cite[Theorem~4]{GK}, we obtain that  for almost all $\vx=(x_1, x_2)$ with respect to $\mu$ 
$$
\sum_{n=1}^{N}\e(x_1n+x_2n^2)=o\(N^{(s_1+s_2)/\rho}\), \quad N\rightarrow \infty.
$$
Since $\mu(\cE_{2,\alpha})>0$,  there is  a set of $(x_1, x_2)$  of positive $\mu$-measure such that 
$$
\left|\sum_{n=1}^{N}\e(x_1n+x_2n^2)\right |\ge N^{\alpha}
$$
for infinitely many $N\in \N$. It follows that 
$$
\alpha\le (s_1+s_2)/\rho.
$$
By taking the concrete parameters we obtain   
$$
t<\min\{1/2+3(1-\alpha), 6(1-\alpha)\}.
$$
Note that this inequality holds for any  $t<\dim \cE_{2,\alpha}$. Thus we obtain 
$$
\dim \cE_{2,\alpha}\le \min\{1/2+3(1-\alpha), 6(1-\alpha)\},
$$
which yields the desired upper bound.

 Proofs of Theorems~\ref{thm:Weyl} and~\ref{thm:F_f} follow a similar idea, albeit in a more technically involved way.

\subsubsection*{Lower bound:} we make the heuristic argument of~\cite[Section~8]{CKMS} rigorous for the case $d=2$. In brief, Gauss sums are large at rational points and their small neighbourhood, and we know the size of them from Diophantine analysis. Indeed, first note that the Gauss sums are large at rational points, for instance for any $a, b, p$ where $p$ is a prime number, $(p, b)=1$ we have  
$$
\left|\sum_{n=1}^{p}\e_p(an+bn^2)\right|=\sqrt{p},
$$
where $\e_p(z) =  \exp(2\pi i z/p)$, 
and hence, by periodicity,  for suitably large $N$ we have 
$$
\left|\sum_{n=1}^{N}\e_p(an+bn^2)\right|\approx \frac{N}{\sqrt{p}}\approx N^{\alpha}.
$$ 
Here  $z\approx Z$ means that  $Z/C\le z\le CZ$ for some absolute positive constant $C$. By the continuity of the Gauss sums, we have 
$$
\left|\sum_{n=1}^{N}\e(xn+yn^2)\right|\approx N^{\alpha},
$$
provided that 
\begin{equation}
\label{eq:Diophantine}
\left|x-\frac{a}{p}\right|<p^{-\frac{1}{2(1-\alpha)}}, \quad \left|y-\frac{b}{p}\right|<p^{-\frac{1}{1-\alpha}}.
\end{equation}

It follows that if $(x, y)$ satisfies~\eqref{eq:Diophantine} for infinitely many $a, b, p$ then $(x, y)\in \cE_{2,\alpha}$, and we denote the collection of these $(x, y)$ by $W_\alpha$. Thus $W_\alpha\subseteq \cE_{2,\alpha}$. 
 By Rynne~\cite[Theorem~1]{Rynne}, for $\alpha\in (1/2, 1)$ we have 
$$
\dim W_\alpha=\min\left \{1/2+3(1-\alpha), 6(1-\alpha)\right \},
$$
and thus yields the desired lower bound.

\section{Results for  sets of exponential sums of arbitrary size}
 \label{sec:arb sums}
 
\subsection{One parametric families of Weyl sums with integer polynomials}

Theorem~\ref{thm:F_f}  says nothing for $\dim \sF_{f, \alpha}$ when  $\alpha \in (1/2, 1-1/D)$. 
Our next result fills that gap, and in particular  for any polynomial $f\in \Z[X]$ with degree $d \ge 2$ and any $\alpha\in (1/2, 1)$ we have 
\begin{equation}
\label{eq:<1}
\dim \cF_{f, \alpha}<1.
\end{equation}
In fact, Theorem~\ref{thm:separated} below implies that for any real polynomial $f\in \R[X]$ with degree $d\ge 2$  the nontrivial bound~\eqref{eq:<1} still holds.

\begin{theorem}
\label{thm:poly}
Let $f\in \Z[x]$ be a polynomial with degree $d\ge 2$.  For any real $\alpha \in (1/2,1)$, we have 
$$
 \dim \cF_{f, \alpha}\le \min\{U_1(d, \alpha), U_2(d, \alpha)\},
$$
 where 
\begin{align*}
&U_1(d, \alpha)=\min_{r=1, \ldots, d}\frac{d+1-\alpha+2^r(1-\alpha)-r}{d+1-\alpha},\\
&U_2(d, \alpha)=\min_{r=1, \ldots, d}\frac{d+1-\alpha+r(r+1)(1-\alpha) - r}{d+1-\alpha}.
\end{align*}
\end{theorem}

Now we extract  some easier  upper bounds for $\dim \cF_{f, \alpha}$.  Taking $r=2$ and $r=d$ in the definition of  $U_1(d, \alpha)$, we obtain
$$
U_1(d, \alpha)\le \min \left \{\frac{d+3-5\alpha}{d+1-\alpha}, \frac{(2^d+1)(1-\alpha)}{d+1-\alpha} \right \}.
$$
Furthermore, taking $r=d$ in the definition of  $U_2(d, \alpha)$ we obtain
$$
U_2(d, \alpha)\le \frac{(d^2+d+1)(1-\alpha)}{d+1-\alpha}.
$$

We formulate the following corollary. 

\begin{cor}
\label{cor:main}
Let $f\in \Z[x]$ be a polynomial with degree $d\ge 2$.  For any real $\alpha \in (1/2,1)$, we have 
$$
\dim \cF_{f, \alpha}\le \min \left\{\frac{d+3-5\alpha}{d+1-\alpha}, \frac{(2^d+1)(1-\alpha)}{d+1-\alpha}, \frac{(d^2+d+1)(1-\alpha)}{d+1-\alpha} \right\}.
$$
\end{cor}

Corollary~\ref{cor:main} implies that for any $f\in \Z[x]$ be a polynomial with degree $d\ge 2$ and any  $\alpha\in (1/2, 1)$ we have 
$$
 \dim \cF_{f, \alpha}<1.
 $$

Furthermore, for monomials we have yet abother  bound. 

\begin{theorem}
\label{thm:monom}
Let   $d\ge 2$.  For any real $\alpha \in (1/2,1)$, we have 
$$
 \dim \sF_{d, \alpha}\le   
  \frac{(1+s_0)(1-\alpha)}{d+1-\alpha},
$$
where 
$$
s_0 = d(d-1)  + \min_{r =1, \ldots, d }\frac{2d + (r-1)(r-2)}{r}. 
$$
\end{theorem}

Furthermore, for ``polynomial-like'' sequences, such  in the special case of  \textit{Piatetski-Shapiro}  sequences $f(n) = \fl{n^\tau}$ 
we have the following result.

\begin{theorem}
\label{thm:P-S}
Let  $f(n) = \fl{n^\tau}$   for some $\tau\ge 1$. 
For any $\alpha\in (1/2, 1)$,  we have 
$$
\dim \cF_{f, \alpha} \le \begin{cases}  
 1-\frac{4\alpha + \tau -4}{\tau+1-\alpha}  & \text{if}\  \tau < 2,\\  
 1-\frac{4\alpha-2}{\tau+1-\alpha} & \text{if}\ \tau  \ge 2. 
 \end{cases}  
$$
\end{theorem}

Note that for any $\alpha\in (1/2, 1)$ Theorem~\ref{thm:P-S} provides nontrivial upper bound, that is,  $\dim \cF_{f, \alpha}<1$, in a wide range of parameters $\alpha$ and $\tau$, for instance, when 
$2 >  \tau > 4 -  4\alpha$. 

\subsection{One parametric families of exponential sums with arbitrary sequences} 
Let $f(n)$ be a real sequence.  We extend the definition of exponential sums $V_f(x; N)$ in~\eqref{eq:sum V}  and of 
the sets $\cF_{f, \alpha}$  to exponential sums with an  arbitrary real sequence  $f(n)$, $n=1,2, \ldots$. 

First we consider sequences  with a given rate of  their growth on average. Namely, we assume 
that there exists a real number $\pow>0$ such that for all large enough $N$ we have 
\begin{equation}
\label{eq:sigma}
\frac{1}{N} \sum_{n=1}^{N} |f(n)| \ll N^{\pow}.
\end{equation}

\begin{theorem} 
\label{thm:separated}
Let $f(n)$ be a real sequence such that $|f(n)-f(m)|\gg 1$ for all $n\neq m$ and the sequence  $f(n)$ satisfies~\eqref{eq:sigma} for some $\pow > 0$.
 For any real $\alpha \in (1/2,1)$, we have 
$$
\dim \cF_{f, \alpha}\le 1-  \frac{2\alpha-1}{\pow+1-\alpha}.
$$
\end{theorem} 

Clearly for any $\alpha\in (1/2, 1)$ and $\pow >0$ we have 
$$
1-  \frac{2\alpha-1}{\pow+1-\alpha}<1.
$$

\begin{theorem}
\label{thm:integer-convex}
Let $f(n)\in \Z$ be a  strictly convex integer sequence that satisfies~\eqref{eq:sigma} for some $\pow>0$. 
Then we have  
$$
\dim \cF_{f, \alpha}\le \min\{U_1(\pow, \alpha), U_2(\pow, \alpha)\}, 
$$
where   
\begin{align*}
& U_1(\pow, \alpha)=\frac{\pow+45/13-5\alpha}{\pow-\alpha},\\
& U_2(\pow, \alpha)=\inf_{k\ge 3, \, k\in \N} \frac{\pow-\alpha+2k-1+2^{-k+1}-2k\alpha}{\pow+1-\alpha}.
\end{align*}
\end{theorem}

Next we consider sequences  with a restriction on the  growth of individual terms rather than 
on average as in~\eqref{eq:sigma}.

\begin{theorem}
\label{thm:O(n^d)}
Let $f(n)$ be a sequence of strictly increasing  sequence of  natural numbers with $f(n)=O(n^{\tau})$ for some $\tau\ge 1$. 
For any $\alpha\in (1/2, 1)$,  we have 
$$
\dim \cF_{f, \alpha} \le 1-\frac{2\alpha-1}{\tau}.
$$
\end{theorem}

We note that that the upper bound of Theorem~\ref{thm:O(n^d)}  is nearly sharp when $\alpha\rightarrow 1$. In fact we consider the following set of exponential sums with an 
even more stringent condition. Namely, for $c > 0$, we define the set 
\begin{equation} 
\label{eq:sG}
\cG_{f, c}=\{x\in T:~|V_f(x; N)|\ge cN  \text{ for infinitely many $N\in \N$ }\}.
\end{equation}

Here we mention some related work on $\cG_{f, c}$.    Suppose further that $f(n)\in \N$ for all $n$ with $f(n)=O(n^{\tau})$  for some $\tau\ge 1$.  
First Salem~\cite{Salem} and then 
Erd{\H o}s and Taylor~\cite{ET}  have shown that the set of $x\in \T$ such that the sequence 
$$
x f(n) , \ n = 1, 2, \ldots, \ \text{is not uniformly distributed}\ \pmod 1
$$
has Hausdorff dimension at most $1-1/\tau$. The result has been extended to arbitrary real sequences  $f(n)=O(n^{\tau})$ by Baker~\cite{B}. Moreover, for each $p\ge 1$, Ruzsa~\cite{Ruzsa} exhibits an integer  sequence $f(n)=O(n^{\tau})$ to show that the above upper bound $1-1/\tau$ is attained. The above results are related to the set $\cG_{f, c}$ by using the \emph{Weyl criterion} (see~\cite[Section~1.2.1]{DrTi}) and \emph{the countable stability} of Hausdorff dimension (see~\cite[Section~2.2]{Falconer}). More precisely, we recall that the countable stability of Hausdorff dimension asserts that for any sequence of sets $\cF_i$ we have 
$$
\dim \bigcup_{i\in \N} \cF_i =\sup_{i\in \N}\dim \cF_i.
$$

It follows from the above result of Baker~\cite{B},  Erd{\H o}s and Taylor~\cite{ET}  and Salem~\cite{Salem}   that for a sequence $f(n)=O(n^{\tau})$ with $\tau\ge 1$ and any $c>0$ we have 
\begin{equation}
\label{eq: G UpperBound}
\dim \cG_{f, c}\le 1-1/\tau.
\end{equation}\
Moreover, the result of Ruzsa~\cite{Ruzsa}  implies that for any $\tau\ge 1$ and any $\varepsilon>0$ there exists a sequence $f(n)=O(n^{\tau})$ such that 
\begin{equation}
\label{eq:vare}
\dim \cG_{f, c}\ge 1-1/\tau-\varepsilon.
\end{equation}
By combining with other ideas, we could remove the $\varepsilon$ of~\eqref{eq:vare}, and obtain the following.

\begin{theorem}  
\label{thm:O(n^d) - lower}
For any $\tau\ge 1$ there exits a strictly increasing sequence of  natural numbers $f(n)$ with $f(n)=O(n^{\tau})$ such that for some $c > 0$,  we have 
$$
\dim\cG_{f, c} = 1-1/\tau.
$$
\end{theorem}

Let $f(n)=O(n^{\tau}), \tau\ge 1$ be a monotone increasing real sequence such that $f(n+1)-f(n)\gg1$. Baker, Coatney and Harman~\cite[Theorem 1]{BCH} show  that  
the set of $\vx\in \T_d$ such that the sequence 
$$
f(n) \vx, \ n = 1, 2, \ldots, \ \text{is not uniformly distributed}\ \pmod 1
$$
is of Hausdorff dimension at most $d-1/\tau$.  Moreover,~\cite[Theorem~2]{BCH} shows that the bound $d-1/\tau$ is sharp.  
By adapting the method of the proof of Theorem~\ref{thm:poly} (which is  same  as in the proofs of Theorems~\ref{thm:monom}, \ref{thm:P-S},  \ref{thm:separated} and~\ref{thm:integer-convex}), we obtain the following result where $f(n)$ is a sequence of matrices.

For a  $d\times d$  matrix $A$ we use  $\|A\|$ to denote its {\it opetator norm\/}, that is, 
$$
\|A\|=\sup \left \{\|\vx A\|:~\vx \in \R^d, \ \|\vx\|=1\right \}
$$
where, as before, $\|w\|$ denotes the Euclidean norm in $\R^{d}$.

\begin{theorem} 
\label{thm:matrix}
Let $\cS=(A_n)_{n\in \N}$ be a sequence of $d\times d$ integer matrices such that for some $\tau \ge 1/d$ and for all $N\in \N$ we have 
\begin{equation}
\label{eq:norm-A}
\frac{1}{N} \sum_{n=1}^{N} \|A_n\| \ll N^{\tau}.
\end{equation} 
Moreover for any $n\neq m$ the matrix $A_n-A_m$ is invertible. Let $\cE_{\cS}$ be the collection of point $\vx\in \T_d$ such that the sequence 
$$
\vx A_n, \ n = 1, 2, \ldots, \ \text{is not uniformly distributed}\ \pmod 1.$$  Then we have 
$$
\dim \cE_{\cS} \le d-1/\tau.
$$
\end{theorem}

We remark that the reason of taking $\tau\ge 1/d$  is making the estimate~\eqref{eq:cardi}  meaningful. Moreover, we claim that the bound of Theorem~\ref{thm:matrix} is sharp in general, and this follows by 
using the aforementioned~\cite[Theorem~2]{BCH}. More precisely,~\cite[Theorem~2]{BCH} shows 
that for any $\tau\ge 1$, there exists an integer sequence $f(n)=O(n^\tau)$ such that the set of $\vx\in \T_d$ 
for which  
$$
\vx f(n), \ n = 1, 2, \ldots, \ \text{is not uniformly distributed}\ \pmod 1
$$ 
is of Hausdorff dimension $d-1/\tau$. For example,  for each $n\in \N$ let $A_n$ be a diagonal matrix with the same diagonal elements $f(n)$, then this sequence $A_n$ attains the above upper bound $d-1/\tau$ which implies the claim above.

\subsection{Comparison} 

Clearly Theorem~\ref{thm:O(n^d)} applies to polynomials $f\in \R[X]$ with $\tau = d$ and thus complements Theorem~\ref{thm:F_f}.  In particular, as we have mentioned  we see that for any real polynomial $f\in \R[X]$ with degree $d\ge 2$  the nontrivial bound~\eqref{eq:<1} still holds.

We remark that Theorem~\ref{thm:separated} still hold under a relaxed condition, that is $|f(n)-f(m)|\gg 1$ for all $n\neq m\ge N_0$ for any  constant $N_0$, 
and thus also   applies to polynomials $f\in \R[X]$ with $\pow = d$, however the corresponding upper bounds implied by 
Theorems~\ref{thm:O(n^d)} and~\ref{thm:separated} satisfy
$$
1-\frac{2\alpha-1}{d} <  1-  \frac{2\alpha-1}{d+1-\alpha}. 
$$

Finally the lower bound in Theorem~\ref{thm:O(n^d) - lower} is based on an idea of Ruzsa~\cite{Ruzsa}.

\subsection{Ideas behind the proofs}   Results of  Section~\ref{sec:arb sums}  are all based on various mean values theorems, continuity of exponential sums and on the completion technique
in the style used in~\cite{ChSh-IMRN, ChSh-JNT, ChSh-TwoPar}.   We roughly show that how their arguments imply the upper bounds of $\dim \cF_{f, \alpha}$. For obtaining the upper bound of  $\dim \cF_{f, \alpha}$, 
we find some intervals to cover the set $\cF_{f, \alpha}$. We collect these intervals by using the continuity of the sums 
$$
V_f(x;N)=\sum_{n=1}^N \e(xf(n)),
$$
that is if $|V_f(x; N)|\ge N^{\alpha}$ then $|V_f(y; N)|\ge N^{\alpha}/2$ when $y$ belongs to some small neighbourhood of  $x$. Moreover, the mean value bounds of $V_f(x; N)$ control the cardinality of above chosen intervals, which yields the desired upper bounds. Thus for obtaining better bounds, we have to know  how small neighbourhood of above $x$ explicitly, and we need various mean values theorems as well. 
 For technical reasons (completion technique), we in fact  use an auxiliary exponential sums $W_f(x; N)$,  which is given by~\eqref{eq:W_f}  to `control' the size of   $V_f(x; N)$.

More precisely, 
to establish Theorem~\ref{thm:poly}, \ref{thm:separated} and~\ref{thm:integer-convex}, we combine  Lemma~\ref{lem:general} 
with  mean values theorems collected in  Section~\ref{subsec:mean}. 
To prove  Theorem~\ref{thm:O(n^d)} we use the mean value bound from  Lemma~\ref{lem:B-18}, which comes from~\cite{Bak-1}, 
and, as for results in Section~\ref{sec:large sums}, on  the 
Frostman  Lemma and 
G\'al--Koksma Theorem,  see Lemmas~\ref{lem:Frostman} and~\ref{lem:G-K for our setting}, respectively.

\section{Frostman  Lemma and G\'al--Koksma Theorem}
\label{sec:FL GKT}
 
\subsection{Frostman  Lemma} 
For a real $s\ge 0$ and a set  $\cF\subseteq \R^{d}$ denote 
\begin{align*}
\cH^{s}_\delta(\cF)=  \inf\biggl\{  \sum_{i=1}^{\infty}\(\diam\cU_i\)^{s}:~ & \cF\subseteq \bigcup_{i=1}^{\infty} \cU_i, \\
&  \cU_i \subseteq \R^{d} \text{ and  }  \diam(\cU_i)\le \delta, i\in \N\biggr\},
\end{align*}
and the $s$-dimensional Hausdorff measure of the set $\cF$ is defined as 
$$
\cH^{s}(\cF)= \lim_{\delta\rightarrow 0} \cH^{s}_{\delta}(\cF).
$$
Moreover, alternatively the Hausdorff dimension of $\cF$ can also be  defined as 
$$
\dim \cF=\inf \{s\ge0:~\cH^{s}(\cF)=0\}=\sup\{s\ge 0:~\cH^{s}(\cF)=\infty\}.
$$

We   also need the following result,  which is known as  the \textit{Frostman  Lemma}, see, for 
instance,~\cite[Corollary~4.12]{Falconer}. 

\begin{lemma}
\label{lem:Frostman}
Let $\cF\subseteq \R^d$ be a Borel set with $0<\cH^{s}(\cF)\le \infty$. Then there exists a compact set $\cE\subseteq \cF$ such that $0<\cH^s(\cE)<\infty$ and 
$$
\cH^s(\cE \cap B(\vx, r))\ll r^{s}  
$$
for all $\vx$ and $r>0$.
\end{lemma}  

We remark that for our application of  Lemma~\ref{lem:Frostman} we  take $\mu$ to be the restriction of  the $s$-Hausdorff measure $\cH^s$ on $\cE$, that is for any $\cA\subseteq \R^d$, 
$$
\mu(\cA) = \cH^s(\cE\cap \cA).
$$
It follows that $\mu$ is a Radon measure such that $\mu(\cE)>0$ and 
$$
\mu(B(\vx, r))\ll r^{s}
$$
for all $\vx$ and $r>0$.

\subsection{G\'{a}l--Koksma Theorem}  

We first recall the following result of G\'{a}l and Koksma~\cite[Theorem~4]{GK}, which shows that  mean value bounds imply almost all bounds with respect to the same measure. Note that~\cite{GK} treats the case of the  Lebesgue measure on a set $\cS$ rather than an arbitrary Radon measure $\mu$, however the proof goes through without change. 

\begin{lemma}
\label{lem:Gal-Koksma}
Let $\mu$ be a Radon measure on $\T_d$ and 
let $f_1(\vx), f_2(\vx), \ldots $ be a sequence of Borel measurable function on $\R^d$. Suppose that we have the bound, for some $\rho>0$ and for all $M\ge 0, N\ge 1$,
$$
\int_{\R^d} \left|  \sum_{n=M+1}^{M+N}f_n(\vx) \right|^{\rho} d\mu(\vx)\le C \Psi(N)\Phi(M, N), \quad \Phi(M, N)\ge 1,
$$
where $C$ is an absolutely constant, $\Psi(N)$ and $\Phi(M, N)$ are some positive functions, and $\Psi(N)/N^{1+\gamma}$ is nondecreasing for some positive $\gamma$.  Let $\psi(N)>0$ be  nondecreasing and 
\begin{equation}
\label{eq:psi}
\psi(2^{n})\ge \Phi(0, 2^{n})+\sum_{\lambda=1}^{n}2^{(\lambda-n)(1+\beta)}\sum_{k=0}^{2^{n-\lambda}-1} \Phi(2^{n}+k2^{\lambda}, 2^{\lambda-1}),
\end{equation}
where $0<\beta<\gamma$ is a constant. Let $\chi(N)$ be a positive nondecreasing function with 
\begin{equation}
\label{eq:convergent}
\sum_{n=1}^{\infty}(N\chi(N))^{-1}<\infty.
\end{equation}
Then for almost all $\vx$ with respect to $\mu$, we have 
$$
\sum_{n=1}^{N}f_n(\vx)=o(\Psi(N)\psi(N)\chi(N))^{1/p}, \quad N\rightarrow \infty.
$$
\end{lemma}

For convenience of our application, we formulate the following particular case of  Lemma~\ref{lem:Gal-Koksma}. 

\begin{lemma}
\label{lem:G-K for our setting}Let $\mu$ be a Radon measure on $\T_d$ and let $f_1(\vx), f_2(\vx), \ldots $ be a sequence of Borel measurable function on $\R^d$. Suppose that we have the bound, for some $\rho>0$ and for all $M\ge 0, N\ge 1$,
$$
\int_{\R^d} \left|  \sum_{n=M+1}^{M+N}f_n(\vx) \right|^{\rho} d\mu(\vx)\le N^{s_1+o(1)}(M+N)^{s_2}
$$
for some   constants $s_1> 1$ and $s_2> 0$. Then for  almost all $\vx\in \R^d$ with respect to $\mu$ we have 
$$
\left | \sum_{n=1}^{N}f_n(\vx) \right |\le N^{(s_1+s_2)/\rho+o(1)}, \qquad N\rightarrow \infty.
$$
\end{lemma}

\begin{proof}
We show that  $\psi(N)=CN^{s_2}$ satisfies~\eqref{eq:psi} when $C$ is some large constant. To see this, 
we note that for $1\le \lambda\le n$,   we have
$$
\sum_{k=0}^{2^{n-\lambda}-1}(2^{n}+k2^{\lambda}+2^{\lambda-1})^{s_2}\ll 2^{n-\lambda}2^{ns_2},
$$
and for some small $\beta>0$,
$$
\sum_{\lambda=1}^{n}2^{(\lambda-n)(1+\beta)}2^{n-\lambda}2^{ns_2} =
\sum_{\lambda=1}^{n}2^{-(n-\lambda)\beta}2^{ns_2} \ll 2^{ns_2}.
$$
Moreover, for any $\varepsilon>0$ the function $\chi(N)=N^{\varepsilon}$ satisfies~\eqref{eq:convergent}.  By Lemma~\ref{lem:Gal-Koksma} we obtain that for almost all $\vx\in \R^d$ with respect to $\mu$ we have 
$$
\left | \sum_{n=1}^{N}f_n(\vx) \right |\le N^{(s_1+s_2+\varepsilon)/\rho+o(1)}, \quad N\rightarrow \infty.
$$
By the arbitrary choice of $\varepsilon>0$ we obtain the desired bound.
\end{proof}

\section{Bounds of Gauss sums and Weyl sums}

\subsection{Structure of large Gauss sums} 
\label{sec:struct} 

The following result of Baker~\cite[Theorem~3]{Bak0} and~\cite[Theorem~4]{Bak1} describes the structure 
of large Gauss sums.  

For each $m\in \N$ recalling that  $[m]=\{0, 1, \ldots, m-1\}$.  
 
\begin{lemma} 
\label{lem:structure of large Gauss} We fix some  $\varepsilon > 0$, and  suppose  that for a real   
$$
A> N^{1/2+ \varepsilon},
$$ 
we have $|G(x_1, x_2;N)|\ge A$ for some $(x_1,x_2)\in \R^2$. 
Then there exist integers $q, a_1, a_2$ such that 
$$
1 \le q \le  \(NA^{-1}\)^2 N^{o(1)},
$$
and for $i=1, 2$ we have 
$$
\left| x_i - \frac{a_i}{q} \right | \le  (NA^{-1})^2 q^{-1}  N^{-i + o(1)}.
$$
\end{lemma} 

We now use Lemma~\ref{lem:structure of large Gauss} to descrbite the structure of large sums $G(x_1, x_2; M, N)$.

\begin{lemma}  
\label{lem:structure of large Gauss-MN} We fix some  $\varepsilon > 0$, and  suppose  that for a real   
$$
A> N^{1/2+ \varepsilon},
$$ 
we have $|G(x_1, x_2;M, N)|\ge A$ for some $(x_1, x_2)\in \T_2$. 
Then there exist integers $q, b_1, b_2$ such that 
$$
1 \le q \le  \(NA^{-1}\)^2 N^{o(1)}, \quad 0\le b_1, b_2 \le q,
$$
and
\begin{align*}
& \left| x_1- \frac{b_1}{q} \right | \le (M+N) q^{-1}A^{-2} N^{o(1)},\\
& \left| x_2- \frac{b_2}{q} \right | \le  q^{-1}A^{-2} N^{o(1)}.
\end{align*}
\end{lemma}

\begin{proof}
Elementary arithmetic shows that 
$$
|G(x_1, x_2; M, N)|=|G(x_1+2Mx_2, x_2; N)|.
$$
By Lemma~\ref{lem:structure of large Gauss} there exist $q, a_1, a_2$ such that 
\begin{equation}
\label{eq:x1x2}
\begin{split}
& \left|x_1+2Mx_2-\frac{a_1}{q} \right | \le q^{-1}A^{-2}N^{1+o(1)},\\
& \left| x_2- \frac{a_2}{q} \right | \le  q^{-1}A^{-2} N^{o(1)}.
\end{split}
\end{equation}

Clearly $A^{-2}N^{o(1)}<1$, we conclude that $0\le a_2\le q$, and we take $b_2=a_2$.

We now search for an integer $b_1$ with the desired property.  From~\eqref{eq:x1x2} we obtain 
\begin{equation}
\label{eq:x_1}
\left| x_1- \frac{a_1-2Ma_2}{q} \right | \le (M+N) q^{-1}A^{-2} N^{o(1)}.
\end{equation}

Suppose that $(M+N)A^{-2} N^{o(1)}\le 1$, then by~\eqref{eq:x_1} we conclude that 
$$
0\le a_1-2Ma_2\le q,
$$
and we take $b_1=a_1-2Ma_2$.  
Suppose to the contrary that 
$$
(M+N)A^{-2} N^{o(1)}> 1,
$$
then trivially there exists $b_1$ with $0 \le b_1\le q$ and such that 
$$
 \left| x_1- \frac{b_1}{q} \right | \le q^{-1}\le 
 (M+N)q^{-1}A^{-2} N^{o(1)},
$$
which finishes the proof. 
\end{proof}

\begin{remark}
For large enough $M$  the bound 
$$
\left| x_1- \frac{b_1}{q} \right | \le (M+N) q^{-1}A^{-2} N^{o(1)}
$$
of Lemma~\ref{lem:structure of large Gauss} would be trivial. Thus we may add $1/q$ term, that is 
$$
\left| x_1- \frac{b_1}{q} \right | \le \min\{ (M+N) q^{-1}A^{-2} N^{o(1)}, 1/q\}.
$$
However, the bound $(M+N) q^{-1}A^{-2} N^{o(1)}$ is sufficient for our applications, and hence we use this bound only. 
\end{remark}

\subsection{Frequency of large Gauss sums}

Let $\mu$ be a Radon measure on $\T_2$ such that 
\begin{equation}
\label{eq:s-Holder}
\mu(B(\vx, r))\ll r^{t}
\end{equation}
holds for some $t>0$ and for all $\vx\in \T_2$ and $r>0$. Let $\fR$ be a rectangle with side length $a<b$, then we have 
$$
\mu(\fR)\ll \min\{ba^{t-1}, b^t\}.
$$
Indeed, we can either include  $\fR$ in a ball of radius $O(b)$ or cover it by $O(b/a)$ balls of radius $a$.

\begin{lemma}
\label{lem:frequency}
Let $\mu$ be a Radon measure satisfying~\eqref{eq:s-Holder} and  $A$ be a real number with $1\le A\le N$. For fixed $M, N\in \N$ denote
$$
\cE_B =\{\vx\in \T_2:~|G(\vx; M, N)|\ge B\}.
$$
Then we have 
$$
\mu(\cE_B)\le N^{6-2t+o(1)}B^{-6} \min \{ M+N, (M+N)^{t}\}.
$$
\end{lemma}

\begin{proof}
Denote 
$$
Q=(NB^{-1})^{2} N^{o(1)}.
$$
From Lemma~\ref{lem:structure of large Gauss} we conclude that 
$$
\{\vx\in \T_2:~ |G(\vx; M, N)|\ge B\}\subseteq \bigcup_{q\le Q}\bigcup_{(a_1, a_2)\in [q]^2} \fR_{q, b_1, b_2},
$$
where $\fR_{q, b_1, b_2}$ is a rectangle with side lengths
$$
(M+N) B^{-2} q^{-1}  N^{ o(1)} \mand   B^{-2} q^{-1}  N^{ o(1)}.
$$  

For $Z\ge 1$, we write $z\sim Z$ to denote that $Z/2< z\le Z$. Denote 
$$
\delta_z= B^{-2} z^{-1} . 
$$  
Taking a dyadic partition of the interval $[1, Q]$, we see that there is a number $1\le Z\le Q$ such that 
\begin{align*}
\mu(\cE_B)&\le N^{o(1)} \sum_{q\sim Z} \sum_{(a_1, a_2)\in [q]^2} \mu(\fR_{q, a_1, a_2})\\
&\le N^{o(1)} Z^3\min\{(M+N)\delta_Z^{t}, (M+N)^{t}\delta_Z^{t}\}\\
&\le N^{o(1)}\min\{I, J\} ,
\end{align*} 
where 
$$
I=  Z^{3}(M+N)\delta_Z^{t}
\le Z^{3-t}B^{-2t}(M+N)\le N^{6-2t }B^{-6}(M+N),
$$
and 
$$
J= Z^{3}(M+N)^{t}\delta_Z^{t}\le N^{6-2t }B^{-6}(M+N)^{t},
$$
which finishes the proof.
\end{proof}

\subsection{Bounds of Weyl sums}

Corresponding to the bounds of Gauss sums, we have the following estimation for Weyl sums which is  needed for the proof of Theorem~\ref{thm:Weyl}.

An integer number $n$ is called
\begin{itemize}
\item  \textit{$r$-th power free} if   any prime number $p \mid n$ satisfies $p^r \nmid n$; 

\item  \textit{$r$-th power  full} if any prime number  $p \mid n$ satisfies $p^r \mid n$. 
\end{itemize}
We note that $1$ is both cube free and cube full.

For any integer $i\ge 2$ it is  convenient to denote   $$
\sQ_{i}=\{n \in \N:~  \text{$n$ is $i$-th power full}\} \quad \text{and} \quad
\sQ_{i}(x)= \sQ_i\cap [1,x].
$$ 
The classical result of  Erd{\H o}s and  Szekeres~\cite{ErdSz} gives an asymptotic 
formula for the cardinality of $\sQ_{i}(x)$ which we present here in a very relaxed form 
as the upper bound 
\begin{equation}
\label{eq:d-full}
\# \sQ_{i}(x) \ll x^{1/i} . 
\end{equation}

The following Lemma~\ref{lem:Struct Large Weyl} comes from~\cite[Lemma 2.7]{BCS}.

\begin{lemma}
\label{lem:Struct Large Weyl}
We fix $d\ge 3$, some  $\varepsilon > 0$, and  suppose  that for a real 
$$
A> N^{1-1/D + \varepsilon}, 
$$ 
where $D$ is given by~\eqref{eq:D}, 
we have $|\sfS_{d}(\vx; N)| \ge A$ for some $\vx\in \R^d$. 
Then there exist positive integers $q_2, \ldots,  q_d$ with  $\gcd(q_i,q_j) = 1$, $2 \le i < j \le d$, and  such that  
\begin{itemize} 
\item[(i)] $q_2$ is  cube  free,
\item[(ii)]  $q_i$ is $i$-th power full but $(i+1)$-th power free when  $3\le i\le d-1$,
\item[(iii)]  $q_d$ is   $d$-th power full, 
\end{itemize}
and
$$
 \prod_{i=2}^d q_i^{1/i} \le N^{1+o(1)}A^{-1} 
 $$
and integers $b_1, \ldots, b_d$ 
such that 
$$
\left|x_j-\frac{b_j}{q_2\ldots q_d}\right |\le (NA^{-1})^{d} N^{-j+o(1)}  \prod_{i=2}^d q_i^{-d/i} , \qquad j=1, \ldots, d.
$$
\end{lemma}

We now need a version of  Lemma~\ref{lem:Struct Large Weyl} for the sums
$$
\sfS_{d}(\vx; M, N)=\sum_{n=1}^{N}\e\(x_1(M+n)+\ldots+x_d(M+n)^{d}\)
$$
 over arbitrary intervals. 

\begin{lemma}
\label{lem:structure of large S} 
We fix $d\ge 3$ and some  $\varepsilon > 0$.
Let $M\ge 0$  and $N\ge 1$. 
Suppose $\vx\in \T_d$ and 
$$
|\sfS_{d}(\vx; M, N)|\ge B \ge N^{1-1/D+\varepsilon}, 
$$
where $D$ is given by~\eqref{eq:D}. 
Then there exists $q=q_2q_3\ldots q_d$ with  $\gcd(q_i,q_j) = 1$, $2 \le i < j \le d$, and  such that 
\begin{itemize} 
\item[(i)] $q_2$ is  cube  free,
\item[(ii)]  $q_i$ is $i$-th power full but $(i+1)$-th power free when  $3\le i\le d-1$,
\item[(iii)]  $q_d$ is   $d$-th power full, 
\end{itemize}
and 
$$
q_2^{1/2}q_3^{1/3}\ldots q_d^{1/d}\le N^{1+o(1)}B^{-1},
$$
and there are $a_1, \ldots, a_d \in [q]$ such that 
\begin{equation}
\label{eq:x_k}
x_k=\frac{a_k}{q}+O\(\(M+N\)^{d-k}r\), \qquad k=1,\ldots, d,
\end{equation}
where
\begin{equation}
\label{eq:r}
r=N^{o(1)}B^{-d}\prod_{i=2}^{d}q_i^{-d/i}.
\end{equation}
\end{lemma}
\begin{proof}
The coefficient of $n^{k}$ in $(M+n)x_1+\ldots +(M+n)^{d}x_d$ is 
\begin{equation}
\label{eq:y_k}
y_k=\sum_{j=k}^{d} \binom{j }{k} M^{j-k}x_j.
\end{equation}
Note that $y_d=x_d$. It follows that 
\begin{align*}
&\left |\sum_{n=1}^{N}\e\((M+n)x_1+\ldots  +(M+n)^{d}x_d\)\right | \\
&\quad \quad \quad\quad\quad\quad=\left |\sum_{n=1}^{N}\e\(y_1n+\ldots+y_dn^{d}\)\right |.
\end{align*}
Then we have the following equivalence 
$$
|\sfS_{d}\(\vx; M, N\)|\ge B\quad \Longleftrightarrow \quad |\sfS_{d}(\vy; N)|\ge B.
$$  
By Lemma~\ref{lem:Struct Large Weyl}  there exist positive integers $q_2 \ldots q_d$ with the above mentioned properties $(i), (ii), (iii)$  and  
$$
 \prod_{i=2}^d q_i^{1/i} \le N^{1+o(1)}A^{-1},
$$
and integers $b_1, \ldots, b_d$ such that 
\begin{equation}
\label{eq:y_j}
\left|y_j-\frac{b_j}{q_2\ldots q_d}\right |\le N^{d-j} r, \qquad j=1, \ldots, d,
\end{equation}
where $r$ is given by~\eqref{eq:r}. 

We now going to show~\eqref{eq:x_k} holds by induction. First note that since $x_d=y_d$, we have the bound 
$$
x_d=\frac{b_d}{q_2\ldots q_d}+O(r).
$$
Suppose that~\eqref{eq:x_k} hold for any $k+1\le j\le d$, that is, there exist $a_j$, $k+1\le j\le d$, such that 
\begin{equation}
\label{eq:induction}
x_j=\frac{a_j}{q}+O\(\(M+N\)^{d-j}r\), \quad j=k+1,\ldots, d.
\end{equation}
Applying~\eqref{eq:y_k},~\eqref{eq:y_j}, we derive that 
$$
\left|x_k +\sum_{j=k+1}^{d}{j \choose k} M^{j-k}x_j -\frac{b_k}{q_2\ldots q_d}\right |\le N^{d-k} r.
$$
Combining with~\eqref{eq:induction} we conclude that there exists $a_k$ such that 
\begin{align*}
\left | x_k- \frac{a_k}{q_2\ldots q_d}\right |& \le N^{d-k}r  +O\(\sum_{j=k+1}^{d}M^{j-k}(M+N)^{d-j}r\)\\
&\ll (M+N)^{d-k}r,
\end{align*}
from which we obtain~\eqref{eq:x_k} by  induction.

Applying similar argument as in the proof of Lemma~\ref{lem:structure of large Gauss-MN}, we can always find $a_1, a_2, \ldots a_d \in [q_2\ldots q_d]$  such that the desired property hold.
\end{proof}

Suppose that $\cR$ is a rectangle with side lengths 
$$
K^{d-1}\delta\ge K^{d-2}\delta\ge \ldots \ge K\delta\ge \delta
$$
for some constants $K\ge 1, \delta>0$. Then, by elementary geometric argument, for any integer $1\le h\le d$ we can cover $\cR$ by $O(K^{h(h-1)/2})$
cubes with side length $K^{d-h}\delta$.  Furthermore, suppose that $\mu$ is a Radon measure satisfying~\eqref{eq:s-Holder}, then we conclude that 
\begin{equation}
\label{eq:mu(R)}
\mu(\cR)\ll K^{h(h-1)/2}\left (K^{d-h}\delta\right )^{t}.
\end{equation}

\begin{lemma}
\label{lem:frequency-d>2} 
We fix $d\ge 3$ and some  $\varepsilon > 0$.
Let $\mu$ be a Radon measure satisfying~\eqref{eq:s-Holder} and
let  $B$ be a real number with $ N^{1-1/D+\varepsilon}\le B\le N$, where $D$ is given by~\eqref{eq:D}. 
For fixed $M, N\in \N$  denote 
$$
\cE_B =\{\vx\in \T_d:~|\sfS_{d}(\vx; M, N)|\ge B\}.
$$
Then for any integer $h$ with $1\le h\le d$ we have 
$$
\mu\(\cE_B\)\le B^{-d^2-1}N^{d^2+1-dt+o(1)}(M+N)^{(d-h)t+h(h-1)/2}.
$$
\end{lemma}

\begin{proof}
Denote  
\begin{equation} \label{eq:def Q}
Q = (NB^{-1})^d, 
\end{equation}
and fix some  $\eta> 0$.  
By Lemma~\ref{lem:structure of large S}, we conclude that 
\begin{equation}
 \label{eq: E_B R}
\cE_B\subseteq \bigcup_{(q_2, \ldots, q_d)\in \Omega}\bigcup_{\va\in [q_2\ldots q_d]^{d}}\cR_{q_2,\ldots, q_d, \va}, 
\end{equation}  
where, slightly relaxing the conditions of Lemma~\ref{lem:structure of large S}, 
we can take
\begin{equation}
\begin{split} 
 \label{eq: Omega}
\Omega =\biggl\{\(q_2, \ldots, q_d\)\in \N^{d-1}:~ q_j \in \sQ_j,  & \ 3 \le j \le d, \\
&   \prod_{j=2}^d q_j^{1/j} \le C Q^{1/d} N^{\eta}  \biggr\}
\end{split} 
\end{equation}
and  $\cR_{q_2,\ldots, q_d, \va}$ is a rectangle with side lengths 
$$
(M+N)^{d-1}r_{q_2,\ldots, q_d}  \ge \ldots \ge (M+N)r_{q_2,\ldots, q_d}  \ge r_{q_2,\ldots, q_d} ,
$$
with  some 
$$
r_{q_2,\ldots, q_d} =N^{o(1)}B^{-d}\prod_{j=2}^{d}q_j^{-d/j}.
$$
Let  $1\le h\le d$ be an integer.  Combining~\eqref{eq: E_B R} with the bound~\eqref{eq:mu(R)}, we obtain 
$$
\mu\(\cE_B\)\ll \sum_{(q_2, \ldots, q_d)\in \Omega} (q_2\ldots q_d)^{d}r_{q_2,\ldots, q_d} ^{t}(M+N)^{(d-h)t+h(h-1)/2}. 
$$
Covering  $\Omega$ by $O\(\(\log N\)^{d-1}\)$ dyadic boxes, we see that 
that there are some integers  $Q_2, \ldots, Q_d \ge 1$ with 
\begin{equation}
\label{eq:prod-Q_j}
\prod_{j=2}^d Q_j^{1/j} \ll Q^{1/d} N^{\eta}
\end{equation}
such that  
\begin{align*}
\mu\(\cE_B\)&\ll \(\log N\)^{d-1}   (M+N)^{(d-h)t+h(h-1)/2} \\
& \qquad  \qquad  \qquad  \qquad \times  \sum_{\substack{q_2 \sim Q_2, \ldots, q_d \sim Q_d\\
q_3\in \sQ_3\(Q_3\), \ldots, q_d\in \sQ_d\(Q_d\)}} (q_2\ldots q_d)^{d} r_{q_2,\ldots, q_d} ^{t}\\
&= N^{o(1)}B^{-dt}(M+N)^{(d-h)t+h(h-1)/2}  \\
& \qquad \qquad \qquad \qquad \qquad  \times \sum_{\substack{q_2 \sim Q_2, \ldots, q_d \sim Q_d\\
q_3\in \sQ_3\(Q_3\), \ldots, q_d\in \sQ_d\(Q_d\)}} \prod_{j=2}^{d}q_j^{d - dt/j}\\
&= N^{o(1)}B^{-dt}(M+N)^{(d-h)t+h(h-1)/2} \\
& \qquad \qquad \qquad \qquad \qquad \times  Q_2^{d - dt/2}   \prod_{j=3}^{d} \(Q_j^{d - dt/j}  \# \sQ_j\(Q_j\)\). 
\end{align*}
Recalling~\eqref{eq:d-full} we derive 
\begin{equation} \label{eq:mu(E_B)}
\mu\(\cE_B\)  \le N^{o(1)}B^{-dt}(M+N)^{(d-h)t+h(h-1)/2} \prod_{j=2}^{d}Q_j^{\alpha_j},
\end{equation}
where 
$$
\alpha_2=d+1-dt/2, \mand  \alpha_j= d+1/j-dt/j, \quad 3\le j\le d.
$$
Observe that for  $j=2, \ldots, d$, we have 
\begin{equation} \label{eq:alpha_i}
 \alpha_j \le    d\alpha_d /j , 
\end{equation}
which for  $j \ge 3$ is obvious from 
 $$
 j \alpha_j= dj +1-dt \le d^2+1-dt =  d\alpha_d
 $$
 and for $j=2$ from 
 $$
2 \alpha_2 = 2d +2  -dt \le d^2+1-dt
 $$
since $d \ge 3$. 
 
Therefore, in view of~\eqref{eq:alpha_i}, recalling~\eqref{eq:prod-Q_j}, we obtain
$$
\prod_{j=2}^{d}Q_j^{\alpha_j}\le \(\prod_{j=2}^{d}Q_j^{1/j}\)^{d\alpha_d} \ll Q^{\alpha_d} N^{d 
\alpha_d \eta}
$$
Then combining this with~\eqref{eq:mu(E_B)} we obtain
$$
\mu\(\cE_B\)  \le N^{o(1)}B^{-dt}(M+N)^{(d-h)t+h(h-1)/2} Q^{\alpha_d} N^{d \eta}.
$$
Recalling the choice of $Q$ in~\eqref{eq:def Q}, 
since $\eta>0$ is arbitrary we obtain the desired bound.
\end{proof}

We need the following analogue of Lemma~\ref{lem:frequency-d>2}.

\begin{lemma} 
\label{lem:frequency-f} We fix  some  $\varepsilon > 0$.
Let $f\in \R[X]$ be a polynomial of degree $d$.
Let $\mu$ be a Radon measure satisfying~\eqref{eq:s-Holder} and 
let  $B$ be a real number with $ N^{1-1/D+\varepsilon}\le B\le N$,  where $D$ is given by~\eqref{eq:D}. For fixed $M, N\in \N$  denote 
$$
\sE_B =\{\vx\in \T:~|V_f(x; M, N)|\ge B\}.
$$
Then 
$$
\mu\(\sE_B\)\le 
N^{o(1)}\begin{cases}  
 N^{4- 2t }   B^{-4}    & \text{if}\  d=2,\\
 N^{d +1 - dt} B^{-d-1} & \text{if}\ d  \ge 3.
 \end{cases} 
$$
\end{lemma}

\begin{proof}  We proceed as in  the proof  of Lemma~\ref{lem:frequency-d>2}. In particular, we fix some  $\eta> 0$
and define $Q$ by~\eqref{eq:def Q}. 

Suppose that 
$$
f(n)=\beta_0+\beta_1n+\ldots+\beta_d n^d,
$$
where $\beta_i\in \R, 0\le i\le d$ and $\beta_d\neq 0$. Since the leading coefficient of $f(n+M)$ coincides with that of $f(n)$, that is $\beta_d$, we see that 
$$
\left |\sum_{n=1}^{N}\e\(xf(n+M)\)\right | =\left |\sum_{n=1}^{N}\e\(y_1n+\ldots+y_{d}n^{d}\)\right |,
$$
where $y_i$, $0\le i\le d$ depend on $M, x$ and in particular $y_d=x\beta_d$.  It follows that if  $|V_f(x; M, N)|\ge B$ for some $x\in \T$ then 
$$
 \left |\sum_{n=1}^{N}\e\(y_1n+\ldots+y_{d}n^{d}\)\right|\ge B,
$$  
where $y_d=\beta_d x$. By Lemma~\ref{lem:Struct Large Weyl} there are $q_2, q_3, \ldots, q_d$, which satisfy the conditions of Lemma~\ref{lem:Struct Large Weyl} and some integer $b$ such that 
$$
\left |\beta_d x-\frac{b}{q_2\ldots q_d}\right|\le N^{o(1)}B^{-d} \prod_{i=2}^{d}q_i^{-d/i},
$$
which is equivalent to 
\begin{equation}
\label{eq:x}
|x-\frac{b}{q_2\ldots q_d \beta_d}|\le N^{o(1)}B^{-d} \prod_{i=2}^{d}q_i^{-d/i}.
\end{equation}
Since $|\beta_d x|\le |\beta_d|$, we derive that 
$$
\frac{|b|}{q_2\ldots q_d}\le |\beta_d|+N^{o(1)}B^{-d} \prod_{i=2}^{d}q_i^{-d/i},
$$
 and thus  
\begin{equation}
\label{eq:range of b}
|b|\le  2 |\beta_d| q_2\ldots q_d
\end{equation}
provided  that $N$ is large enough. It follows that for  large enough $N$ we have 
\begin{equation}
 \label{eq: SS_B R}
\sE_B\subseteq \bigcup_{(q_2, \ldots, q_d)\in \Omega}\bigcup_{ \substack{b\in \Z\\ |b|\le 2 |\beta_d| q_2\ldots q_d }}\cI_{q_2,\ldots, q_d, b}, 
\end{equation}
where, as in the proof  of Lemma~\ref{lem:frequency-d>2}, the set $\Omega$ is given by~\eqref{eq: Omega} 
and  $\cI_{q_2,\ldots, q_d, b}$ is an interval of  length 
$$
|\cI_{q_2,\ldots, q_d, b}|\le N^{o(1)}B^{-d}\prod_{j=2}^{d}q_j^{-d/j}.
$$
Hence we derive from~\eqref{eq:s-Holder},~\eqref{eq:x},~\eqref{eq:range of b} and~\eqref{eq: SS_B R}, that 
\begin{align*}
\mu\(\sE_B\) & \le  N^{o(1)}B^{-dt}  \sum_{(q_2, \ldots, q_d)\in \Omega} \prod_{j=2}^{d}q_j \(\prod_{j=2}^{d}q_j^{-d/j}\)^t\\
& \le  N^{o(1)}B^{-dt}  \sum_{(q_2, \ldots, q_d)\in \Omega} \prod_{j=2}^{d}q_j^{1-dt/j}. 
\end{align*}
Again  as in the proof  of Lemma~\ref{lem:frequency-d>2}, covering  $\Omega$ by $O\(\(\log N\)^{d-1}\)$ dyadic boxes, we see that 
that there are some integers  $Q_2, \ldots, Q_d \ge 1$ with~\eqref{eq:prod-Q_j}
such that  
\begin{align*}
\mu\(\sE_B\)&\le  N^{o(1)}B^{-dt}  
  \sum_{\substack{q_2 \sim Q_2, \ldots, q_d \sim Q_d\\
q_3\in \sQ_3\(Q_3\), \ldots, q_d\in \sQ_d\(Q_d\)}}  \prod_{j=2}^{d}q_j^{1-dt/j}\\
&\le  N^{o(1)}B^{-dt}   \prod_{j=2}^{d}Q_j^{1-dt/j} \sum_{\substack{q_2 \sim Q_2, \ldots, q_d \sim Q_d\\
q_3\in \sQ_3\(Q_3\), \ldots, q_d\in \sQ_d\(Q_d\)}} 1\\
&\le  N^{o(1)}B^{-dt} Q_2^{2 - dt/2}   \prod_{j=3}^{d} \(Q_j^{1 - dt/j}  \# \sQ_j\(Q_j\)\). 
\end{align*}
Thus, by~\eqref{eq:d-full}, we have 
\begin{equation}
\label{eq:prelim 1}
\mu\(\sE_B\) \le  N^{o(1)}B^{-dt} Q_2^{2 - dt/2}   \prod_{j=3}^{d} Q_j^{1 - (dt-1)/j} . 
\end{equation}
Denote 
$$
\prod_{j=2}^d Q_j^{1/j}  = R. 
$$
Then we can rewrite~\eqref{eq:prelim 1} as 
\begin{equation}
\label{eq:prelim 2}
\mu\(\sE_B\) \le  N^{o(1)}B^{-dt}  R^{ - dt+1}  Q_2^{3/2}    \prod_{j=3}^{d} Q_j. 
\end{equation}
If $d=2$ then $Q_2 = R^2$ and~\eqref{eq:prelim 2} becomes 
$$
\mu\(\sE_B\) \le  N^{o(1)}B^{-2t}  R^{4 - 2t}  .
$$ 
Using $R \ll Q^{1/2} N^{\eta}$ and recalling the definition of $Q$ in~\eqref{eq:def Q} we obtain, 
\begin{equation}
\label{eq:fin d=2}
\mu\(\sE_B\) \le  N^{4- 2t + \eta(4- 2t)+o(1)}   B^{-4}  .
\end{equation}

Now let  $d \ge 3$, then trivially
$$
 \prod_{j=3}^{d} Q_j \le  \prod_{j=3}^{d} Q_j^{d/j}  =\(RQ_2^{-1/2}\)^d.
$$
Hence we derive from~\eqref{eq:prelim 2} that 
$$
\mu\(\sE_B\) \le  N^{o(1)}B^{-dt}  R^{d - dt+1}  Q_2^{3/2-d/2 } \le   N^{o(1)}B^{-dt}  R^{d - dt+1}   .
$$
Since $t \le 1$ we have $d-dt +1 >0$. Therefore, using $R \ll Q^{1/d} N^{\eta}$ we obtain
\begin{equation}
\label{eq:fin d>2}
\begin{split}
\mu\(\sE_B\) & \le  N^{ \eta(2d - dt+1)}B^{-dt}  \(NB^{-1}\)^{d - dt+1}\\
& =  N^{d - dt+1 +  \eta(d - dt+1)}B^{-d-1} .
\end{split}
\end{equation}

Since $\eta$ is arbitrary, from~\eqref{eq:fin d=2} and~\eqref{eq:fin d>2}  we derive the desired result. 
\end{proof}

\section{Proof of the upper bound of Theorem~\ref{thm:Gauss}}

\subsection{Mean values of Gauss sums}
We need the following mean value estimate for Gauss sums with respect to an arbitrary Radon measure, 
which is interesting in its own right.

\begin{lemma}
\label{lem:Lp}
Let $\mu$ be a Radon measure on $\T_2$ such that 
$$
\mu(B(\vx, r))\ll r^{t}
$$
holds for some $t>0$ and for all $\vx\in \T_2$ and $r>0$. Then   for all $M, N$ we have 
$$
\int_{\T_2} \left |G(\vx; M, N) \right |^{6} d\mu(\vx)\le N^{6-2t+o(1)}(M+N)^{ \min\{1, t\}} .
$$
\end{lemma}

\begin{proof} Let us fix some $\varepsilon>0$.  Denote 
$$
K=N^{1/2+\varepsilon}.
$$
By a dyadic partition argument, there exits $B\in [K, N]$ such that
\begin{equation}
 \label{eq:Dyadic split}
 \int_{\T_2} \left |G(\vx; M, N) \right |^{6} d\mu(\vx) \le K^6 \mu(\T_2) +  B^6\mu(\cE_B) N^{o(1)} .
\end{equation} 
By Lemma~\ref{lem:frequency} we have 
$$
\mu(\cE_B)\le N^{6-2t+o(1)}B^{-6} (M+N)^{ \min\{1, t\}}.
$$
which after substitution in~\eqref{eq:Dyadic split} implies 
$$
 \int_{\T_2} \left |G(\vx; M, N) \right |^{6} d\mu(\vx) \le N^{3 + 6 \varepsilon}   + N^{6-2t+o(1)}(M+N)^{ \min\{1, t\}} .
 $$
Since  $t\le 2$ and $\varepsilon$ is arbitrary, the result now follows.  
\end{proof}

\subsection{Concluding the proof} 
We now turn to the proof of the upper bound of Theorem~\ref{thm:Gauss}. Let $t\in (0, \dim \cE_{2,\alpha})$. Then $\cE_{2, \alpha}$ has infinite $\cH^{t}$-measure. By Lemma~\ref{lem:Frostman}, there exists a Radon measure $\mu$ on $\T_2$ with 
$$
\mu(\cE_{2, \alpha})>0 \quad \text{ and } \quad  \mu(B(\vx, r))\ll r^{t}  
$$
for all $\vx\in \T_2$ and $r>0$.  Taking the function 
$$
f_n(x_1, x_2)=\e(x_1n+x_2n^2)
$$
and applying Lemmas~\ref{lem:G-K for our setting} and~\ref{lem:Lp}, we immediately derive that for  
almost all $(x_1, x_2)\in \T_2$ with respect to $\mu$,  
\begin{equation}
\label{eq:point-wise}
|G(x_1, x_2; N)|\le N^{1-t/3+\min\{1/6, t/6\}+o(1)}.
\end{equation}
Since $\mu(\cE_{2, \alpha})>0$,  there is a set of $(x_1, x_2) \in \T_2$  of positive $\mu$-measure 
such that 
$$
|G(x_1, x_2; N)|\ge N^{\alpha}
$$
for infinitely many $N\in \N$. Combining with~\eqref{eq:point-wise} we derive 
$$
\alpha\le 1-t/3+\min\{1/6, t/6\},
$$
which implies 
$$
t\le \min\{1/2+3(1-\alpha), 6(1-\alpha)\}.
$$
Since this holds for any $t<\dim \cE_{2, \alpha}$, we conclude that 
$$
\dim \cE_{2, \alpha} \le \min\{1/2+3(1-\alpha), 6(1-\alpha)\},
$$
which yields the desired upper bound.

\section{Proof of the lower bound of   Theorem~\ref{thm:Gauss}} 

\subsection{Large values of Gauss sums}
The main purpose of this subsection is to show  Lemma~\ref{lem:continuity 2}.  We start from  recalling the following property of Gaussian sums, 
see~\cite[Equation~(1.55)]{IwKow}.

\begin{lemma}
\label{lem:Gauss}
Let $p\ge 3$ and $a, b\in \Z_p$ with $b\neq 0$, then 
$$
\left|\sum_{n=0}^{p-1}\ep\left(an+bn^{2}\right)\right|=\sqrt{p}.
$$
\end{lemma}

Using the Gauss bound together with the standard 
completion technique, see~\cite[Sections~11.11 and~12.2]{IwKow} we also immediately obtain: 

\begin{lemma}
\label{lem:incompleteMonom}
For any  prime  $p$ and any $a \in \F_p\setminus\{0\}$ we have 
$$
\max_{1\le M, N\le p}\left|\sum_{M+1\le n\le M+N}\ep\left(an^{2}\right)\right| \ll \sqrt{p} \log p.
$$
\end{lemma} 

The continuity of Gauss sums yields the following result.

\begin{lemma}
\label{lem:continuity}
For $N\gg p\log p$ we have 
\begin{align*}
|G(x_1, x_2; N)&-G(a/p, b/p; N)|\\
& \ll Np^{-1/2}(|x_1-a/p|N+|x_2-b/p|N^{2}).
\end{align*}
\end{lemma}

\begin{proof}
We use the following version of summation by parts.
Let $a_n$ be a sequence and for each $t\ge 1$ denote   
$$
A(t)=\sum_{1\le n\le t} a_n.
$$
Let $\psi: [1, N]\rightarrow \C$ be a continuously differentiable function. Then 
$$
\sum_{n=1}^{N} a_n \psi(n)=A(N)\psi(N)-\int_{1}^{N} A(t)\psi'(t)dt.
$$
Let $\delta_1=x-a/p$, $\delta_2=x_2-b/p$. Then define 
\begin{equation}
\label{eq:Delta=G-G}
\begin{split}
\Delta&=G(x_1, x_2; N)-G(a/p, b/p; N)\\
&=\sum_{n=1}^{N}\e(na/p+n^2b/p) \left (\e(\delta_1n+\delta_2n^{2})-1 \right ). 
\end{split}
\end{equation}

For an integer $M$ with $1\le M\le N$, we split the sum $G(a/p, b/p;M)$ into $O(N/p)$ complete sums and at most one 
incomplete sum,  applying  Lemmas~\ref{lem:Gauss} and~\ref{lem:incompleteMonom}, we derive
$$
\max_{1\le M\le N} |G(a/p, b/p;M)| \ll Np^{-1/2}+p^{1/2}\log p.
$$  
Hence,  applying to the sum in~\eqref{eq:Delta=G-G} summation by parts with $a_n=\e(na/p+n^2b/p) $  and $\psi(t)=\e(\delta_1t+\delta_2t^{2})-1$, we derive 
 that 
\begin{align*}
\Delta&\ll \max_{1\le M\le N} |G(a/p, b/p;M)| \(|\delta_1|N+|\delta_2| N^{2}\)\\
&\ll \(Np^{-1/2}+p^{1/2}\log p\)\(|\delta_1|N+|\delta_2| N^{2}\)\\
&\ll Np^{-1/2}\(|\delta_1|N+|\delta_2| N^{2}\),
\end{align*}
which finishes the proof.
\end{proof}

From Lemma~\ref{lem:Gauss} and Lemma~\ref{lem:continuity} we  obtain the following, which is the main purpose  of this subsection.

\begin{lemma}
\label{lem:continuity 2}
We fix $\alpha\in (1/2, 1)$. Let $p\ge 3$ and $a, b\in \Z_p$ with $b\neq 0$. Let $N$ be the smallest  number such that  $p| N$ and 
$$
N\ge p^{\frac{1}{2(1-\alpha)}}.
$$
Then there exists a sufficiently small number $\eta>0$ such that for any $(x_1, x_2)\in \T_2$ with 
$$
\left |x_1-\frac{a}{p}\right |<\eta p^{-\frac{1}{2(1-\alpha)}} \mand \left |x_2-\frac{b}{p}\right |<\eta p^{-\frac{1}{1-\alpha}}
$$
we have 
$$
G(x_1, x_2; N)\gg N^{\alpha}.
$$
\end{lemma}

\begin{proof}
Recalling that   $z\approx Z$ means  $Z/C\le z\le CZ$ for some absolute positive constant $C$. 
First note that the choice of $N$ implies  
$$
N\approx p^\frac{1}{2(1-\alpha)}\mand Np^{-1/2}\approx N^{\alpha}.
$$   
By  Lemma~\ref{lem:continuity} we have 
$$
|G(x_1, x_2; N)-G(a/p, b/p; N)|\ll \eta N^{\alpha}.
$$
Lemma~\ref{lem:Gauss} implies $G(a/p, b/p; N)\approx N^{\alpha}$. Therefore, we obtain the desired bound by choosing 
a  sufficiently small $\eta$.
\end{proof}

\subsection{Simultaneous Diophantine approximations}
Let $\btheta=(\vartheta_1, \ldots, \vartheta_d)$ be a  vector of positive real numbers and let $\fQ$ an arbitrary set of positive 
integers.  Without losing generality, assuming that $\vartheta_1\le \ldots \le \vartheta_d$. We denote by  $W_{\fQ, \btheta}$ be the  collection of points $(x_1, \ldots, x_d)\in \T_d$ for which there are infinitely many  $q\in \fQ$ such that 
$$
\left \|qx_i \right \|<q^{-\vartheta_i}, \qquad i=1, \ldots, d,
$$
where $\left \| x \right \|=\min\{|x-n|:~n\in \Z\}$.    Denote
$$
 \nu(\fQ) = \inf\left\{\nu \in \R:~ \sum_{q\in \fQ} q^{-\nu} < \infty\right\}. 
$$
We  need the following result of  Rynne~\cite[Theorem~1]{Rynne}. 

\begin{lemma}
\label{lem:Rynne}
Suppose that 
$
\vartheta_1+\ldots +\vartheta_d \ge \nu(\fQ),
$
then we have
$$
\dim W_{\fQ, \btheta}=\min_{1\le j\le d}\frac{d+\nu(\fQ) +j\vartheta_j-\sum_{i=1}^{j}\vartheta_i}{1+\vartheta_j}.
$$
\end{lemma}

For $d=2$ and $ \nu(\fQ)  = 1$ we have the following.
$$
\dim W_{\fQ, \btheta}=\min\left\{ \frac{3}{1+\vartheta_1}, \frac{3+\vartheta_2-\vartheta_1}{1+\vartheta_2}\right \}.
$$

We now turn to the proof of the lower bound of   Theorem~\ref{thm:Gauss}. Indeed this follows by combining Lemma~\ref{lem:continuity 2} and Lemma~\ref{lem:Rynne}. Let  $\fQ$ be the collection of prime numbers. Clearly we have $\nu(\fQ)=1$. By Lemma~\ref{lem:continuity 2}, for any $\varepsilon>0$ we obtain 
$$
W_{\fQ, \vartheta} \subseteq \cE_{2, \alpha-\varepsilon}
$$
with 
$$
\vartheta_1=\frac{1}{2(1-\alpha)}-1 \mand  \vartheta_2=\frac{1}{1-\alpha}-1.
$$
Since $\alpha\in (1/2, 1)$, we have $\vartheta_1+\vartheta_2\ge 1$. Thus by Lemma~\ref{lem:Rynne} we obtain 
$$
\dim W_{\fQ, \vartheta}=\min\{1/2+3(1-\alpha), 6(1-\alpha)\}.
$$
It follows that 
$$
\dim \cE_{2, \alpha-\varepsilon}\ge \min\{1/2+3(1-\alpha), 6(1-\alpha)\}.
$$
Since  $\varepsilon>0$  is arbitrary,  we obtain the desired bound.

\section{Proof of Theorem~\ref{thm:Weyl}}

\subsection{Mean values of Weyl sums}
We  need the following mean value estimate of Weyl sums with respect to a general measure. 

\begin{lemma}
\label{lem:Lp-d>2}
Let $\mu$ be a Radon measure on $\T_d, d\ge 3$ such that 
$$
\mu(B(\vx, r))\ll r^{t}
$$
holds for some $t>0$ and for all $\vx\in \Tor$ and $r>0$. Then for any integer $h$ with $1\le h\le d$ we have 
\begin{align*}
\int_{\T_d}  |\sfS_{d}(\vx; M, N) |^{d^2+1}& d\mu(\vx)\le N^{(1-1/D)(d^2+1)+o(1)}\\
&+N^{d^2+1-dt+o(1)}(M+N)^{(d-h)t+h(h-1)/2},
\end{align*}
 where $D$ is given by~\eqref{eq:D}. 
\end{lemma}
\begin{proof}
Let us fix some $\varepsilon>0$. Denote 
$$
K=N^{1-1/D+\varepsilon}.
$$
Similar to the proof of Lemma~\ref{lem:Lp}, taking a dyadic partition of the interval $[K, N]$, there exists a number $B\in [K, N]$  such that 
\begin{align*}
&\int_{\T_d} |\sfS_{d}(\vx; M, N)|^{d^2+1}d\mu(\vx)\\
&\qquad \le K^{d^2+1}+N^{o(1)}B^{d^2+1}\mu(\{\vx\in \T_d: B\le \sfS_{d}(\vx; M, N)|\le 2 B\}).
\end{align*}
Combining with Lemma~\ref{lem:frequency-d>2} we obtain 
\begin{align*}
&\int_{\T_d} |\sfS_{d}(\vx; M, N)|^{d^2+1}d\mu(\vx)\\
&\qquad \le  N^{(1-1/D)(d^2+1)+\varepsilon(d^2+1)}+N^{o(1)}N^{d^2+1-dt}(M+N)^{(d-h)t+h(h-1)/2}.
\end{align*}
By the arbitrary choice of  $\varepsilon>0$ we obtain the desired bound.  
\end{proof}

\subsection{Concluding the proof} 
We  now turn to the proof of  Theorem~\ref{thm:Weyl}. Let $t\in (0, \dim \cE_{d, \alpha})$. Then we see that $\cE_{d,\alpha}$ has infinite $\cH^{t}$-measure. By Lemma~\ref{lem:Frostman}, there exists a Radon measure $\mu$ on $\T_d$ with 
$$
\mu(\cE_{2, \alpha})>0 \quad \text{ and } \quad  \mu(B(\vx, r))\ll r^{t}  
$$
for all $\vx\in \T_d$ and $r>0$.  

Let $1\le h\le d$ be an integer.  There are two cases to consider.

\textbf{Case~1.} Suppose that the `total exponent'  
$$
d^2+1-dt + (d-h)t+h(h-1)/2 = d^2+1-ht+h(h-1)/2
$$
in  the second term in the bound of  Lemma~\ref{lem:Lp-d>2} is at least as large as the exponent of the first term, that is,  
$$
d^2+1-ht+h(h-1)/2 \ge (1-1/D)(d^2+1).
$$ 
Since $t< \dim \cE_{d, \alpha} \le d$, we have $d^2+1-dt>1$, and thus  by Lemma~\ref{lem:G-K for our setting} we derive that for almost all $\vx\in \T_d$ with respect to $\mu$,
\begin{equation}
\label{eq:almost all-upper}
|\sfS_{d}(\vx;N)|\le N^{\frac{d^2+1-ht+h(h-1)/2}{d^2+1}+o(1)}.
\end{equation}  
Since $\mu(\cE_{d, \alpha})>0$,   there is  a set of $\vx\in \Tor$  of positive $\mu$-measure such that 
$$
|\sfS_{d}(\vx; N)|\ge N^{\alpha}
$$
for infinitely many $N\in \N$. Combining with~\eqref{eq:almost all-upper} we derive 
\begin{equation}
\label{eq:alpha bound 1}
\alpha\le \frac{d^2+1-ht+h(h-1)/2}{d^2+1},
\end{equation}   
which implies 
\begin{equation}
\label{eq:t<h}
t\le \frac{(d^2+1)(1-\alpha)}{h}+\frac{h-1}{2}.
\end{equation}

\textbf{Case~2.} Suppose  that $d^2+1-th+h(h-1)/2<(1-1/D)(d^2+1)$. Then Lemma~\ref{lem:Lp-d>2} implies that 
$$
|\sfS_{d}(\vx;N)|\le N^{d^2+1-dt+o(1)}(M+N)^{(1-1/D)(d^2+1)-(d^2+1-dt)}.
$$
By Lemma~\ref{lem:G-K for our setting} we conclude that for   almost all $\vx\in \T$ with respect to $\mu$, 
$$
|\sfS_{d}(\vx; N)|\le N^{1-1/D+o(1)}.
$$
Then applying the similar argument to  \textbf{Case 1}, we obtain 
\begin{equation}
\label{eq:alpha bound 2}
\alpha\le 1-1/D,
\end{equation}  
which contradicts our assumption that $\alpha\in (1-1/D,1)$.
Thus we are in   \textbf{Case 1} and we have~\eqref{eq:t<h}  for any integer $1\le h\le d$. 
Since~\eqref{eq:t<h} holds for any $t<\dim \cE_{d, \alpha}$, we obtain the desired upper bound.

\section{Proof of Theorem~\ref{thm:F_f}}

\subsection{One-dimensional mean values of Weyl sums}
For the proof of Theorem~\ref{thm:F_f}, similarly to the proofs of the upper bounds of   
Theorems~\ref{thm:Gauss} and~\ref{thm:Weyl},  we see from Lemma~\ref{lem:Frostman} 
that it is sufficient to prove the following mean value bounds.

We start with quadratic polynomials.

\begin{lemma}
\label{lem:mean-d=2} Let $f\in \R[X]$ be a polynomial of degree $d=2$.
Let $\mu$ be a Radon measure on $\T$ such that 
$$
\mu(B(\vx, r))\ll r^{t}
$$
holds for some $t\in (0, 1)$ and for all $x\in \T$ and $r>0$. Then for all $M, N$ we have 
$$
\int_{\T}   |V_f(x; M, N) |^4 d\mu(x)\le N^{4(1-t/2)+o(1)}.
$$
\end{lemma}

\begin{proof}
Let us fix some $\varepsilon>0$. Denote 
$$
K=N^{1/2+\varepsilon}.
$$
Similar to the proofs of Lemmas~\ref{lem:Lp} and~\ref{lem:Lp-d>2},
  taking a dyadic partition of the interval $[K, N]$, there exists a number $B\in [K, N]$  such that
$$
\int_{\T}   |V_f(x; M, N) |^4  d\mu(x)\le K^{4}+N^{o(1)}B^{4}\mu(\cE_{B}). 
$$ 
Hence by  Lemma~\ref{lem:frequency-f} applied with $d=2$, we have
$$
\int_{\T}   |V_f(x; M, N) |^{4}  d\mu(x) \le K^4 + N^{4-2t+o(1)}.
$$
Since $t \le 1$ and  $\varepsilon$ is arbitrary, the result now follows. 
\end{proof}

For polynomials of higher degree we have a similar bound. 

\begin{lemma}
\label{lem:mean-d>2} Let $f\in \R[X]$ be a polynomial of degree $d\ge 3$.
Let $\mu$ be a Radon measure on $\T$ such that 
$$
\mu(B(\vx, r))\ll r^{t}
$$
holds for some $t\in (0, 1)$ and for all $x\in \T$ and $r>0$. Then   for all $M, N$ we have 
$$
\int_{\T}  |V_f(x; M, N) |^{d+1} d\mu(x)\le N^{d+1  - (d+1)/D+ o(1)}  + N^{d+1- dt+o(1)}, 
$$
where $D$ is given by~\eqref{eq:D}.
\end{lemma}

\begin{proof}
Let us fix some $\varepsilon>0$. Denote 
$$
K=N^{1-1/D+\varepsilon}. 
$$
 Then, similarly to the above,  by Lemma~\ref{lem:frequency-f} with $d\ge 3$ there exists $K\le B\le N$ such that 
$$
\int_{\T}   |V_f(x; M, N) |^{d+1} d\mu(x)  \le K^{d+1}+N^{o(1)}B^{d+1}\mu(\cE_{B}).
$$
By Lemma~\ref{lem:frequency-f}, we have
$$
\int_{\T}   |V_f(x; M, N) |^{d+1} d\mu(x)  \le K^{d+1} + N^{d-dt+1+o(1)}.
$$
Since  $\varepsilon$ is arbitrary, the result now follows.  
\end{proof}

\subsection{Concluding the proof} 
Similarly to the proofs of the upper bounds of   
Theorems~\ref{thm:Gauss} and~\ref{thm:Weyl}, Lemmas~\ref{lem:mean-d=2} and~\ref{lem:mean-d>2} combined 
with  Lemma~\ref{lem:G-K for our setting}
imply the desired upper bound of Theorem~\ref{thm:F_f}.

In particular,  for $d=2$ the proof is a full analogue of that of   the upper bound of Theorem~\ref{thm:Gauss} 
where we use Lemma~\ref{lem:mean-d=2} in an appropriate place instead of Lemma~\ref{lem:Lp}. 

For $d\ge 3$, as in the proofs of   
Theorem~\ref{thm:Weyl}, we consider two cases 
$$
t\le  \frac{d+1} {dD} \mand t >  \frac{d+1} {dD} .
$$
Now, by   Lemma~\ref{lem:G-K for our setting}, in the first case, similarly to~\eqref{eq:alpha bound 1},  we derive 
$$
\alpha \le \frac{d-dt+1}{d+1} 
$$
which gives the desired bound, while 
in  the second case we obtain~\eqref{eq:alpha bound 2},   which contradicts the 
assumption $\alpha\in (1-1/D, 1)$.

\section{Proofs of  Theorems~\ref{thm:poly}, \ref{thm:monom}, \ref{thm:P-S},  \ref{thm:separated} and~\ref{thm:integer-convex}}

\subsection{Mean values of exponential polynomials}
\label{subsec:mean}

For a real sequences $f(n)$ we define the sums
\begin{equation}
\label{eq:W_f}
W_{f}( \vx; N)=  \sum_{h=-N}^{N} \frac{1}{|h|+1} \left| \sum_{n=1}^{N}   \e\(h n/N+xf(n) \) \right|.
\end{equation} 
Then a special form of~\cite[ Lemma~3.2]{ChSh-IMRN} implies 
for $\vx\in \T$ and $1\le M\le N$ we have 
\begin{equation}
\label{eq:completion}
V_{f}( \vx; M) \ll W_{f}( \vx; N), 
\end{equation}
where $W_f(\vx; N)$ is given by~\eqref{eq:W_f}.

Our method is based on mean  value estimates on the sums $W_f(x; N)$. 
However our next result shows that   for  integer-valued sequences the mean value of $W_f(x; N)$ is controlled by the mean value of $V_f(x; N)$ provided that the exponent is some even integer. It follows by using a similar argument to 
the proof of~\cite[Lemma~2.4]{ChSh-TwoPar}.   We give a  proof here for completeness.

\begin{lemma} 
\label{lem:S-W mean}
Let $f\in \Z[x]$ be an integer sequence such that for some even number $s>0$ and some real $t>0$ one has 
$$
\int_{\T} |V_f(x;N)|^{s}dx\ll N^{t},
$$
then we have   
$$
\int_{\T} W_f(x;N)^{s}dx\ll N^{t} (\log N)^{s}.
$$
\end{lemma}

\begin{proof}
Write 
\begin{align*}
W_{f}( \vx; N)=  \sum_{h=-N}^{N}\left (\frac{1}{|h|+1}\right )^{1-1/s}&\left (\frac{1}{|h|+1}\right )^{1/s}\\
& \times \left| \sum_{n=1}^{N}   \e\(h n/N+xf(n) \) \right|.
\end{align*}
Applying the H\"older inequality,  we obtain 
\begin{equation}
\label{eq:WS}
\begin{split}
W_f(\vx; N)^{s}\ll \left (\log N\right )^{s-1} &\sum_{h=-N}^{N}\frac{1}{1+|h|} \\
&\quad \quad 
\quad \times \left| \sum_{n=1}^{N}   \e\(h n/N+xf(n) \) \right|^{s}.
\end{split}
\end{equation}
For any $h$ and $N$ and even number  $s$, opening the integral and applying the orthogonal property of $\e(x)$, we derive 
\begin{align*}
 \int_{\T} & \left| \sum_{n=1}^{N}   \e\(h n/N+xf(n) \) \right|^{s} dx \\
&=\#\left \{(n_1, \ldots, n_s):~1\le n_i\le N, \ \sum_{i=1}^{s/2} (f(n_i)-f(n_{s/2+i})=0 \right \}\\
&=\int_{\T} \left| \sum_{n=1}^{N}   \e\(xf(n) \) \right|^{s} dx.
\end{align*}
Combining with~\eqref{eq:WS} we obtain the desired bound.
\end{proof}

Observe that Lemma~\ref{lem:S-W mean} implies that for integer-values sequences, 
we only need to estimate the moments of the sums  $V_f(x; N)$ rather than of  $W_f(x; N)$.

We now recall some   mean value estimates on the sums $V_f(x; N)$ in~\eqref{eq:sum V} when $f\in \Z[x]$ is a polynomial.
We first recall the following  result of Hua~\cite{Hua}, see also~\cite[Section~14]{Wool}. 
\begin{lemma}
\label{lem:Hua}
Let $f\in \Z[x]$ be a polynomial with degree $d\ge 2$, then for each natural number $1\le r\le d$,
$$
\int_{T} |V_f(x;N)|^{2^r}dx\le N^{2^r-r+o(1)}.
$$
\end{lemma}

Wooley~\cite[Corollary 14.2]{Wool} (see also~\cite[Theorem~10]{Bourg}) obtains the following better bound when $r$ is large.

\begin{lemma}
\label{lem:BW}
Let $f\in \Z[x]$ be a polynomial with degree $d\ge 2$, then for each natural number $1\le r\le d$,
$$
\int_{T} |V_f(x;N)|^{r(r+1)}dx\le N^{r^2+o(1)}.
$$
\end{lemma}

Furthermore, for the case of monomials Wooley~\cite[Corollary 14.7]{Wool} 
gives a stronger result. 

\begin{lemma}
\label{lem:Monom-W}
Let   $d\ge 2$ and  
$$
s_0=d(d-1)  + \min_{r =1, \ldots, d }\frac{2d + (r-1)(r-2)}{r}. 
$$
Then
$$
\int_{T}  \left|\sum_{n=1}^{N}\e\(xn^d\)\right|^{s_0}  dx\le N^{s_0-d+o(1)}.
$$
\end{lemma}

We now turn to mean value theorems for sums with arbitrary sequences.

In particular, we have the following simple bound on the second moment of the sums $W_f(x; N)$,  defined by~\eqref{eq:sum V}  with well-spaced sequences.

\begin{lemma}
\label{lem:mean-separated}
Let $f(n)$ be a real sequence such that $f(n)-f(m)\gg 1$ for all $m\neq n$. Then for any interval $\cI$ we have 
$$
\int_{\cI}W_f(x; N)^2 dx \ll N (\log N)^3,
$$
where $W_f(\vx; N)$ is given by~\eqref{eq:W_f}.
\end{lemma}

\begin{proof}
By the Cauchy inequality, we obtain 
\begin{align*}
W_f(x; N)^2& \ll \log N\sum_{h=-N}^{N} \frac{1}{|h|+1} \left| \sum_{n=1}^{N}   \e\(h n/N+xf(n) \) \right|^2\\
&\ll   \log N\sum_{h=-N}^{N} \frac{1}{|h|+1} \left ( N+\Sigma(x) \right ),
\end{align*}
where 
$$
\Sigma(x) =\sum_{1\le n\neq m\le N} \e(h(n-m)/N+x(f(n)-f(m)).
$$
The condition   $|f(n)-f(m)|\gg 1$ implies that the values $f(1), \ldots, f(N)$ are separated
from each by a unit interval, and thus so are  the values $f(1)-\zeta, \ldots, f(N)-\zeta$ for any real $\zeta$.
In particular  
$$
\sum_{\substack{n =1\\ f(n)\ne \zeta}}^N \frac{1}{|f(n)-\zeta|} \ll \log N. 
$$ 
Therefore,
\begin{align*}
\int_{\cI}\left | \Sigma(x)\right | dx & \ll \sum_{1\le n\neq m\le N} \left|\int_{\cI} \e( x(f(n)-f(m))dx \right| \\
& \ll 
\sum_{1\le n\neq m\le N} \frac{1}{|f(n)-f(m)|}\ll N\log N.
\end{align*}
Combining with~\eqref{eq:W_f} we obtain the desired bound.
\end{proof}

Suppose that $f(n)\in \N$ is a strictly convex sequence, Iosevich,   Konyagin,  Rudnev, and  Ten~\cite[Equation~(1.13)]{IKRT} (general even number $s$) and Shkredov~\cite[Theorem~1.1]{Shkredov} ($s=4$) gives the following bounds.

\begin{lemma}
\label{lem:convex}
Suppose that $f(n)\in \N$ is a strictly convex sequence, then  
$$
\int_{\T} |V_f(x;N)|^4dx\ll N^{32/13+o(1)}
$$
and for any even number $s\ge 6$ we have 
$$
\int_{\T} |V_f(x;N)|^sdx\ll N^{s-2+2^{1-s/2}}.
$$
\end{lemma}

We note that for sequences satisfying stronger conditions than convexity stronger versions of Lemma~\ref{lem:convex} are known, see~\cite{BHR}.

We now observe that  the   result of Robert and Sargos~\cite[Theorem~2]{RoSa} implies 
the following bound. 

\begin{lemma}
\label{lem:PS-AddEnergy}
Let  $f(n) = \fl{n^\tau}$   for some $\tau\ge 1$. 
$$
\int_{\T} |V_f(x;N)|^4dx\ll N^{2+o(1)}+ N^{4-\tau+o(1)}. 
$$
\end{lemma}

\subsection{Continuity of exponential polynomials}
The following is a special form of~\cite[Lemma 2.4]{ChSh-JNT}.   
\begin{lemma} 
\label{lem:cont} 
Let $f(n)$ be a real sequnce that satisfies~\eqref{eq:sigma} for some  constant $\pow>0$. Let $0<\alpha<1$ and let $\varepsilon>0$ be  sufficiently small. If 
$W_{f}(x; N)\ge N^{\alpha}$ for  some $x\in \T$,  then  
$$
W_{f}( y; N)\ge  N^{\alpha}/2
$$
holds for  any $y\in (x-\zeta, x+\zeta)$ provided  that $N$ is large enough and 
$$
0<\zeta\le N^{\alpha-\pow -1 -\eps}.
$$
\end{lemma}

\begin{proof}
Note that for any $x, y\in \R$ and any $h, N$ we have 
$$
\e(hn/N+xf(n))-\e(hn/N+yf(n))\ll |x-y||f(n)|.
$$
Thus we obtain 
$$
|W_f(x; N)-W_f(y;N)|\ll |x-y|\log N \sum_{n=1}^{N} |f(n)|,
$$
which yields the desired result.
\end{proof}

\begin{lemma}
\label{lem:cover}
Let $f(n)$ be a real sequnce that satisfies~\eqref{eq:sigma} for some  constant $\pow>0$. Suppose that
$$
\int_{\T} W_f(x; N)^{s} dx\le N^{t+o(1)}.
$$
Then 
$$
\left\{x\in \T:~W_f(x; N)\ge N^{\alpha}\right\}\subseteq \bigcup_{I\in \cI_N} I, 
$$
where $\cI_N$ is a collection of intervals with equal length such that 
$
|I|\le  N^{\alpha-\pow- 1-\varepsilon}
$
for each $I\in \cI_N$ and of cardinality
$$
\# \cI_N\le N^{t-s\alpha+\pow+1-\alpha+2\varepsilon}
$$ 
provided  that $N$ is sufficiently large.
\end{lemma}

\begin{proof}
Let 
$$
\zeta=1/ \rf{N^{\pow+1+\eps-\alpha}}.
$$
We divide $\T$ into 
$ \zeta^{-1} $ intervals of the type $ [k\zeta, (k+1)\zeta]$ with $k=0, 1, \ldots, \zeta^{-1}-1$. Let $\cD_N$ be the collection of these intervals and 
$$
\cI_N=\{I\in \cD_N:~\exists\, x\in I \text{ such that } W_f(x; N)\ge N^{\alpha}\}.
$$
Lemma~\ref{lem:cont}  implies that  
for each $I\in \cI_N$,
$$
W_f(x; N)\ge N^{\alpha}/2, \qquad \forall x\in I.
$$
It follows that 
$$
(\# \cI_N) \zeta N^{s\alpha} \ll \int_{\T} W_f(x; N)^{s}dx\le N^{t+o(1)},
$$
which gives the desired result.
\end{proof}

\subsection{Mean values and Hausdorff dimension for polynomially growing sequences} 
We have the following general result about the upper bound on $\dim \cF_{f, \alpha}$.

\begin{lemma}
\label{lem:general}
Let $f(n)$ be a real sequence that satisfies~\eqref{eq:sigma} for some  constant $\pow>0$. Suppose that there are positive constants $s, t$  such that 
$$
\int_{\T} |W_f(x; N)|^{s} dx\le N^{t+o(1)}.
$$
Then 
$$
\dim \cF_{f, \alpha} \le \frac{\pow+1-\alpha+t-s \alpha}{\pow+1-\alpha}.
$$ 
\end{lemma}

\begin{proof}  For each $N\in \N$ denote 
$$
B_N=\{x\in T:~W_f(x; N)\ge N^{\alpha}\}.
$$
Let  $N_i=2^i$, $i \in \N$ and $\eta>0$.  Applying~\eqref{eq:completion} we obtain
$$
\cF_{f, \alpha+\eta}\subseteq \bigcap_{q=1}^{\infty} \bigcup_{i=q}^{\infty} B_{N_i}.
$$ 
Indeed, let $x\in \cF_{f, \alpha+\eta}$ and suppose to the contrary that    
$$
x\notin \bigcap_{q=1}^{\infty} \bigcup_{i=q}^{\infty} B_{N_i}.
$$
Then  $W_f(x; N_i)<N_i^{\alpha}$ holds for all large $N_i$. For any $N$ there 
 exists a number  $i\in \N$ such that $N_i\le N<N_{i+1}$, and for large enough $N$ by Lemma~\ref{lem:cover} we have 
$$
V_f(x; N)\ll W_f(x; N_{i+1})\ll N_{i+1}^{\alpha}\ll N^{\alpha},
$$
which contradicts our assumption.

For each $N_i$ by Lemma~\ref{lem:cover} we obtain 
$$
B_{N_i}\subseteq \bigcup_{I\in \cI_{N_i}} \cI_{N_i},
$$
where $|I|\le N^{\alpha-\pow -1-\varepsilon}$ for each $I\in \cI_{N_i}$ and $\# \cI_{N_i}\le N_i^{t-s\alpha+\pow+1-\alpha+2\varepsilon}$.

From the definition of the Hausdorff dimension, using the above notation,  we have the following inequality 
\begin{equation}\label{eq:hausdorff}
\dim \cF_{f, \alpha+\eta} \le \inf \left \{ \nu>0:~\sum_{i=1}^{\infty} \sum_{I\in \cI_{N_i}}     |I|^{\nu} <\infty\right \}.
\end{equation}

Note that 
$$
\sum_{i=1}^{\infty} \sum_{I\in \cI_{N_i}}     |I|^{\nu}\ll \sum_{i=1}^{\infty} N_i^{t-s\alpha+\pow+1-\alpha+2\varepsilon} N_i^{(\alpha-\pow-1)\nu},
$$
thus to make the series  convergent it is sufficient to have 
$$
\nu>\frac{t-s\alpha+\pow+1-\alpha+2\varepsilon}{\pow+1-\alpha}.
$$
Combining with~\eqref{eq:hausdorff} and  the arbitrary choice of $\varepsilon>0$ we obtain 
$$
\dim \cF_{f, \alpha+\eta}\le \frac{\pow+1-\alpha+t-s \alpha}{\pow+1-\alpha}.
$$
Since this  holds for any $\eta>0$, we obtain the desired bound.
\end{proof}

\subsection{Concluding the proofs}
Combining  Lemma~\ref{lem:general} (and taking $\pow=d$ for polynomial 
sequneces)   with 

\begin{itemize}
\item Lemmas~\ref{lem:S-W mean}, \ref{lem:Hua} and~\ref{lem:BW}, we obtain Theorem~\ref{thm:poly}; 
\item Lemmas~\ref{lem:S-W mean} and~\ref{lem:Monom-W}, we obtain Theorem~\ref{thm:monom};
\item    Lemmas~\ref{lem:S-W mean} and~\ref{lem:PS-AddEnergy}, we obtain Theorem~\ref{thm:P-S};
\item   Lemma~\ref{lem:mean-separated}, we obtain Theorem~\ref{thm:separated}; 
\item Lemmas~\ref{lem:S-W mean} and~\ref{lem:convex}, we obtain Theorem~\ref{thm:integer-convex}.
\end{itemize}

\section{Proofs of Theorems~\ref{thm:O(n^d)} and~\ref{thm:O(n^d) - lower}}

\subsection{Proof of Theorem~\ref{thm:O(n^d)}} 
The upper bound of Theorem~\ref{thm:O(n^d)} follows by applying Lemmas~\ref{lem:Frostman} and~\ref{lem:G-K for our setting} and  the following mean value bound of Baker~\cite[Equation~(18)]{Bak-1}.

\begin{lemma}
\label{lem:B-18}
Let $f(n)$ be a sequence of natural numbers with $f(n)=O(n^{\tau})$ for some real number $\tau>0$. Let $\mu$ be a Radon measure on $\T$ such that 
$$
\mu(B(\vx, r))\ll r^{t}
$$
holds for some $t\in (0, 1)$ and for all $x\in \T$ and $r>0$. Then   for all $M, N$ we have 
$$
\int_{\T}  |V_f(x; M, N) |^{2} d\mu(x)\ll N (M+N)^{\tau(1-t)}.
$$
\end{lemma}

\subsection{Proof of Theorem~\ref{thm:O(n^d) - lower}}  
By~\eqref{eq: G UpperBound} the upper bound holds always, and hence we now prove the lower bound. For any $\tau \ge 1$, Ruzsa~\cite{Ruzsa} defines a strictly increasing sequence of natural numbers $g(n)$ with $g(n)=O(n^{\tau})$, a constant $0<\ell_0<1$ (depending on $\tau$) and a set 
$\sG \subseteq \T$ with $\dim \sG = 1-1/\tau$, having the following property. 

If $x \in \sG$ then there are infinitely many $N$ and corresponding intervals $\cI(N,x) \subseteq T$ of length $\ell_0$ such that 
\begin{equation}\label{eq:gnx}
\sum_{\substack{n=1\\ \{x g(n)\} \in \cI(N,x)}}^N 1 \ge 2 \ell_0 N,
\end{equation}
where as usual $\{u\}$ denotes the fractional part of a real $u$. 

Fix an integer 
\begin{equation}\label{eq:large H}
H > 4/\ell_0.
\end{equation}  
By~\cite[Lemma~2.7]{Bak-DI}, there is a trigonometric polynomial
$$
\varPsi_{N,x}(y) = C_0 + \sum_{0 < |k| \le H} C_k(N,x) \e(ky), 
$$
depending on $N$ and $x$ with 
\begin{equation}\label{eq:C0}
C_0 = \ell_0 + \frac{1}{H+1}
\end{equation}
and such that 
\begin{equation}\label{eq:fNxy}
\varPsi_{N,x}(y)  \ge \begin{cases}  1 & \text{if}\ y\in \cI(N,x) ,\\
0 & \text{otherwise} .
 \end{cases} 
\end{equation}
Note that since $\varPsi_{N,x}(y)  \ge 0$, for any $k\neq 0$ we have
$$
|C_k(N,x)=\left | \int_{\T}\varPsi_{N,x}(y)\e(-ky) dy \right | \le \int_{\T}\varPsi_{N,x}(y) dy=C_0,
$$
and hence 
\begin{equation}\label{eq:small C_k}
|C_k(N,x)| \le C_0, \qquad 0 < |k| \le H.
\end{equation}

Let $x \in \sG$. Then by~\eqref{eq:gnx} and~\eqref{eq:fNxy}  we have
$$
\sum_{ n=1}^N \varPsi_{N,x}\(xg(n)\) \ge 2\ell_0 N. 
$$
Thus 
$$
C_0N  + \sum_{0 < |k| \le H}\left |C_k(N,x)\right| \left| \sum_{ n=1}^N \e\(kx g(n\))\right|\ge 2 \ell_0 N. 
$$
Recalling~\eqref{eq:large H} and~\eqref{eq:C0} we see that $2\ell_0-C_0\ge \ell_0/2$. Hence 
$$
 \sum_{0 < |k| \le H} \left|C_k(N,x)\right | \left|\sum_{ n=1}^N \e\(kx g(n)\)\right|\ge \frac{\ell_0}{2} N, 
$$
which together with~\eqref{eq:small C_k}  implies that there exists  a number  $k \in \{\pm 1, \ldots, \pm H\}$ such that 
$$
 \left| \sum_{ n=1}^N \e\(kx g(n)\) \right|\ge \frac{\ell_0}{4HC_0}N.
 $$
 (we note that  the number $k$ depends on $x$  and $N$). 
 
Since for any $x\in \sG$ and $N$ such that~\eqref{eq:gnx} holds there are  finite choices of $k$, we conclude that for any $x\in \sG$ there exists a number $k\in \{\pm 1, \ldots, \pm H\}$
such that for infinitely many $N$ we have 
$$
 \left| \sum_{ n=1}^N \e\(kx g(n)\) \right|\ge \frac{\ell_0}{4HC_0}N.
 $$
That is, with  $f_k(n) = |k| g(n)$ we have
 $$
 \sG \subseteq \bigcup_{k=1}^H  \sG_{f_k, c},
 $$
 where $c=\ell_0/(4HC_0)$ and $\sG_{f, c}$ is given by~\eqref{eq:sG}.
 Therefore, for  some $k \in \{1, \ldots, H\}$ we have.
 $$
  \dim \sG_{f_k, c} \ge \dim \sG \ge 1-1/\tau.
$$
Together with the upper bound~\eqref{eq: G UpperBound}
this finishes the proof.

\section{Proof of Theorem~\ref{thm:matrix}}

\subsection{Preliminaries} 
Let $\cS=(A_n)_{n\in \N}$ be a sequence of $d\times d$  matrix. For any $\vh\in \R^d$ (which   treat  as a column vector)  and $N\in \N$ let
$$
V_{\cS, \vh}(\vx; N)=\sum_{n=1}^{N} \e\(\left \langle  \vx A_n, \vh \right\rangle\). 
$$
where $\left\langle  \vy, \vz \right\rangle$ denotes the standard scalar product. 

\begin{lemma}
\label{lem:mean-m}
Let $\cS=(A_n)_{n\in \N}$ be a sequence of $d\times d$ integer matrix such that $(A_n-A_m)$ is invertible if $n\neq m$. Then for any $\vh\in \Z^d\setminus \{{\bf 0}\}$ and $N\in \N$ we have 
$$
\int_{\T_d} \left |V_{\cS, \vh}(\vx; N)\right |^{2} d\vx=N.
$$
\end{lemma}

\begin{proof}
Opening the square we have 
\begin{align*}
 \int_{\T_d}&\left|\sum_{n=1}^{N} \e\(\left\langle \vx A_n, \vh \right\rangle\)\right |^{2} d\vx
\\
&=\sum_{1\le n, m\le N} \int_{\T_d}  \e\(\left\langle \vx, (A_n-A_m)\vh \right\rangle\) d\vx\\
&=N+\sum_{1\le n\neq  m\le N} \int_{\T_d}  \e\(\left\langle \vx, (A_n-A_m)\vh \right\rangle\) d\vx.
\end{align*}
 By our condition that $A_n-A_m$ is invertible when $n\neq m$, we conclude that 
$(A_n-A_m)\vh$ is a non-zero integer vector, and hence for $n\neq m$ we have
$$
\int_{\T_d}  \e\(\left\langle \vx, (A_n-A_m)\vh \right\rangle\) d\vx=0,
$$
which yields the desired identity.
\end{proof}

We have the following analogy of Lemma~\ref{lem:cont}.

\begin{lemma}
\label{lem:con-m} Let $\cS=(A_n)_{n\in \N}$ be a sequence of $d\times d$ matrices that satisfies~\eqref{eq:norm-A} for some $\tau\ge 1/d$ and  let $\vh\in \R^d, c\in (0, 1)$. Then there exists $\varepsilon>0$ such that  if $|V_{\cS,  \vh}(\vx; N)|\ge cN$ for some $\vx\in \T_d$ then 
$$
|V_{\cS, \vh}(\vx; N)|\ge cN/2
$$
holds for any $\vy\in B(\vx, \varepsilon N^{-\tau})$,  where $B(\vx, r)$ denotes the ball of $\T_d$  centered at  $\vx$ and of radius $r$.
\end{lemma}

\begin{proof}
For any $n\in \N, \vh\in \R^d$ and $\vx, \vy\in \T_d$ we have 
\begin{align*}
\e(\left\langle \vx A_n, \vh \right\rangle )-\e(\left\langle \vy A_n, \vh \right\rangle)& 
\ll \left\langle \vx A_n, \vh \right\rangle-\left\langle \vy A_n, \vh \right\rangle\\
& = \left\langle (\vx-\vy) A_n, \vh \right\rangle\ll \|A_n\| \|\vh\|\|\vx-\vy\|.
\end{align*}
It follows that
$$
V_{\cS, \vh}(\vx; N)-V_{\cS, \vh}(\vy; N)\ll  \|\vx-\vy\|  N^{\tau+1}\|\vh\|,
$$
which yields the desired bound.
\end{proof}

\begin{lemma}
\label{lem:level-m} 
Let $\cS=(A_n)_{n\in \N}$ be a sequence of $d\times d$ matrix that satisfies~\eqref{eq:norm-A} for some $\tau\ge 1/d$ and  let $\vh\in \R^d, c\in (0, 1)$.  Then there exists $\varepsilon>0$ such that 
$$
\{\vx\in \T_d:~|V_{\cS, \vh}(\vx; N)|\ge cN\} \subseteq \bigcup_{Q\in \cQ_{N}} Q, 
$$
where $\cQ_N$ is a certain collection of  equal  cubes with the side lengths $1/\fl{N^{\tau} \varepsilon^{-1}}$ and 
\begin{equation}
\label{eq:cardi}
\#\cQ_N\ll N^{d\tau-1},
\end{equation}
where the implied constant depends on $\varepsilon$.
\end{lemma}

\begin{proof} 
Divide $\T_d$ into  $\zeta^{-d}$ interior disjoint equal cubes in a natural way such 
that each cube has side length $\zeta= 1/\fl{N^{\tau}\varepsilon^{-1}}$, and let $\cD_n$ 
be a collection of these cubes. 
Let 
$$
\cQ_N=\{Q\in \cD_n:~\exists\,  \vx\in Q, \ \text{such that}\   |V_{\cS, \vh}(\vx; N)|\ge cN\}.
$$
It is sufficient to show that $\cQ_N$ satisfies~\eqref{eq:cardi}. For any $Q\in \cQ_N$ by  Lemma~\ref{lem:con-m} we have  $|V_{\cS, \vh}(\vx; N)|\ge cN/2$ for all $\vx\in Q$. 
Hence 
$$
N^2 \# \cQ_N  N^{-d\tau}\varepsilon^{d}\ll \int_{\cQ_N}  |V_{\cS, \vh}(\vx; N)|^2 d \vx \le  \int_{\T^d}  |V_{\cS, \vh}(\vx; N)|^2 d \vx. 
$$ 
Combining with the mean value bound Theorem~\ref{lem:mean-m} we derive 
$$
N^{2} N^{-d\tau} \# \cQ_N\ll N,
$$
which implies the desired bound.
\end{proof}

We remark that the condition $\tau\ge 1/d$ in Lemma~\ref{lem:level-m} comes from the
inequality~\eqref{eq:cardi}.

\subsection{Concluding the proof}
We now turn to the proof of Theorem~\ref{thm:matrix}.  Let $\cS=(A_n)_{n\in \N}$ satisfy the condition of Theorem~\ref{thm:matrix}.  For $c>0$ and $\vh\in \R^d$ define 
$$
\cG_{\cS, \vh, c}=\{\vx\in \T_d: |V_{\cS, \vh}(\vx; N)|\ge cN \text{ for infinitely many } N\in \N\}.
$$
By using the \emph{Weyl criterion}  (see~\cite[Section~1.2.1]{DrTi}) and \emph{the countable stability} of Hausdorff dimension (see~\cite[Section~2.2]{Falconer}), it is sufficient to prove that for any $c>0$ and any non-zero vector $\vh\in \Z^{d}$ one has 
$$
\dim \cG_{\cS, \vh, c} \le d-1/\tau.
$$

For $N\in \N$ denote 
$$
B_N=\{\vx\in \T_d: V_{\cS, \vh}(\vx; N)|\ge cN/2\}.
$$
Let $\beta>1$ and $N_i=i^{\beta}$. Then we have 
\begin{align*}
\cG_{\cS, \vh, c}\subseteq \bigcap_{k=1}^{\infty}\bigcup_{i=k}^{\infty} B_{N_i}.
\end{align*}
Indeed let $\vx\in \cG_{\cS, \vh, c}$ and suppose to the contrary that for all large enough $N_i$ we have 
$$
|V_{\cS, \vh}(\vx; N_i)|<cN_i/2.
$$
For any large $N$ there is $i\in \N$ such that 
$N_i\le N<N_{i+1}$. Observe that  $N_{i+1}-N_i=O(i^{\beta-1})$ for all $i\in \N$, and 
\begin{align*}
|V_{\cS, \vh}(\vx; N)|&\le |V_{\cS, \vh}(\vx; N_i)|+N_{i+1}-N_i\\
&\le cN_i/2+O(i^{\beta-1})\le 2c N/3
\end{align*} 
provided that $N$ is large enough, which contradicts our assumption that $\vx\in \cG_{\cS, \vh, c}$. 

Let $\varepsilon>0$ be the same on as in Lemma~\ref{lem:level-m}.  For each $N_i$ by Lemma~\ref{lem:level-m} we obtain 
$$
B_{N_i}\subseteq \bigcup_{Q\in \cQ_{N_i}} Q,
$$
where each $Q\in \cQ_{N_i}$ has side length $1/\fl{N^{\tau} \varepsilon^{-1}}$ and $\# \cQ_{N_i}\ll N_i^{d\tau-1}$.
From the definition of Hausdorff dimension we obtain 
\begin{equation}
\label{eq:upper-m}
\dim \cG_{\cS, \vh, c}\le \inf  \left \{\nu>0: \sum_{i=1}^{\infty}\sum_{Q\in \cQ_{N_i}} N_i^{-\tau\nu}<\infty \right \}.
\end{equation}
Note that 
$$
\sum_{i=1}^{\infty}\sum_{Q\in \cQ_{N_i}} N^{-\tau\nu} \le \sum_{i=1}^{\infty} N_{i}^{d\tau-1}N_i^{-\tau\nu},
$$
thus the series is convergent provided 
$$
\beta(d\tau-1-\tau\nu) <-1,
$$
which is equivalent to 
$$
\nu>d-\frac{1}{\tau}+\frac{1}{\beta \tau}.
$$
Combining with~\eqref{eq:upper-m} we obtain 
$$
\dim \cG_{\cS, \vh, c}\le d-\frac{1}{\tau}+\frac{1}{\beta \tau},
$$
and by the arbitrary choice of  $\beta>1$ we obtain the desired bound.

%%%%%%%

\section*{Acknowledgement}

During preparation of this work,  I.S. was  supported   by ARC Grant DP170100786.

 \end{document}